\documentclass[12pt]{amsart}
\usepackage{amssymb}
\usepackage{amsmath}
\usepackage{amscd}
\theoremstyle{plain}
\newtheorem{theorem}{Theorem}[section]
\newtheorem{corollary}[theorem]{Corollary}
\newtheorem{proposition}[theorem]{Proposition}
\newtheorem{lemma}[theorem]{Lemma}
{\theoremstyle{remark}

\newtheorem{remark}[theorem]{Remark}}
{\theoremstyle{definition}
\newtheorem{definition}[theorem]{Definition}
}

\newcommand{\N}{\mathbb{N}}
\newcommand{\Z}{\mathbb{Z}}

\newcommand{\C}{\mathbb{C}}
\newcommand{\T}{\mathbb{T}}

\newcommand{\cA}{{\mathcal A}}

\newcommand{\ip}[2]{\langle{#1},{#2}\rangle}

\newcommand{\Ca}{$C^*$-al\-ge\-bra }
\newcommand{\CA}{$C^*$-al\-ge\-bra}

\newcommand{\shom}{$*$-ho\-mo\-mor\-phism }

\newcommand{\rom}{\renewcommand{\labelenumi}{{\rm (\roman{enumi})}}%
\renewcommand{\itemsep}{0pt}}
\DeclareMathOperator{\id}{id}

\DeclareMathOperator{\Per}{P}

\DeclareMathOperator{\Orb}{Orb}

\begin{document}
\title[Ideal structure of $C^*$-algebras of SGDS]%
{Ideal structure of \boldmath{$C^*$}-algebras of singly generated dynamical systems}
\author[Takeshi KATSURA]{Takeshi KATSURA}
\address{Department of Mathematics, Keio University,
Yokohama, 223-8522 JAPAN}
\email{katsura@math.keio.ac.jp}

\subjclass[2000]{Primary 46L05; Secondary 46L55, 37B99}

\keywords{}

\begin{abstract}
In this paper, we show that the set of 
all ideals of the $C^*$-algebras of a singly generated dynamical system 
corresponds bijectively to 
the set of all subsets of 
the product of the space of the system and the circle 
 satisfying three conditions. 

\end{abstract}

\maketitle

\setcounter{section}{-1}

\section{Introduction}

In the theory of \CA s, 
it is very important yet very difficult to 
list all ideals of a given \CA . 
In this paper, we list all ideals of a \Ca of 
a singly generated dynamical system 
in terms of subsets of a certain space. 

In \cite{R2}, 
Renault introduces 
the notion of a singly generated dynamical system (SGDS) 
and associates a \Ca with it. 
In \cite{K1}, 
the author introduces 
the notion of a topological graph 
and associates a \Ca with it. 
In \cite{K2}, 
the author shows that an SGDS can be considered as 
a topological graph and their \CA s coincide. 
Conversely in \cite{Ksh} every \Ca of a topological graph 
is a \Ca of some SGDS (see also \cite{Yee} and \cite{KL}). 
Thus the class of \CA s of SGDSs is fairly large. 
Although in \cite{R2} spaces in SGDSs are assumed to be second countable, 
there is no such restriction in \cite{K1}. 
In this paper, spaces are not assumed to be second countable, 
and hence the associated \CA s are not necessarily separable. 

The purpose of this paper is to investigate the ideal structure 
of \CA s of SGDSs. 
Since we do not assume second countability, 
it is almost impossible to list primitive ideals of 
the associated \Ca in terms of points of the space. 
Instead in this paper, we list all ideals directly 
in terms of subsets of the product of the space and the circle 
(Theorem~\ref{MainThm}). 
The investigation here may be useful to describe 
the topology of primitive ideal space 
when the space is second countable 
(cf.\ \cite{SW}). 

This paper is organized as follows. 
In Section~\ref{Sec:C*S}, 
we introduce SGDSs and their \CA s. 
To define a \Ca of an SGDS, 
we do not use groupoids as in \cite{R2}, 
but use a definition similar to \cite{K1}. 
We examine the structure of a \Ca of an SGDS 
which we need later. 
In Section~\ref{Sec:irrep}, 
we construct lots of irreducible representations 
of a \Ca of an SGDS. 
In Section~\ref{Sec:gii}, 
we associate a subset of a space to each ideal, and an ideal to each subset. 
Then, we list all gauge-invariant ideals in terms of invariant subsets 
(Proposition~\ref{Prop:bij}). 
In Section~\ref{Sec:prime}, 
we list all prime ideals, and show that 
the set of primitive ideals constructed in Section~\ref{Sec:irrep} 
is sufficiently large (Corollary~\ref{Cor:inter}). 
In Section~\ref{Sec:IYI}, 
we associate a subset of the product of a space and the circle 
to each ideal, and an ideal to each subset. 
We show that ideals of the \Ca of an SGDS 
correspond injectively to a subset of the product of a space and the circle 
(Proposition~\ref{Prop:IYI}). 
In Section~\ref{Sec:YIad}, 
we show that this subset is an admissible set which is defined in 
Definition~\ref{Def:adm} (Proposition~\ref{Prop:YIinv}). 
In Section~\ref{Sec:YIY}, we show that the set of 
all ideals of the \Ca of an SGDS corresponds bijectively to 
the set of all admissible sets (Theorem~\ref{MainThm}). 

We try to make this paper self-contained as well as possible. 
However, in Section~\ref{Sec:prime} we need two kinds of 
uniqueness theorems for \CA s of topological graphs 
which are quoted in Appendix~\ref{Sec:UT}. 

\vspace{0.5cm}
\noindent
{\bfseries Acknowledgments.} 
This work was supported by JSPS KAKENHI Grant Number JP18K03345.

\section{Singly generated dynamical systems}\label{Sec:C*S}

\begin{definition}
A pair $\Sigma=(X,\sigma)$ is a 
{\em singly generated dynamical system} (in short SGDS) 
if $X$ is a locally compact space 
and $\sigma$ is a local homeomorphism 
from an open subset $U$ of $X$ to $X$. 
\end{definition}

It should be emphasized that $X$ is not necessarily second countable. 
Throughout this paper, 
$\Sigma=(X,\sigma)$ means an SGDS. 
We denote by $\N := \{0,1,2,\ldots\}$ the set of natural numbers. 

\begin{definition}
We define a decreasing sequence $\{U_n\}_{n\in\N}$ 
of open subsets of $X$ by $U_0=X$, $U_1=U$ and 
$U_{n+1}=\{x\in U_n\mid \sigma^{n}(x)\in U\}$. 
\end{definition}

The open set $U_n$ is the domain of $\sigma^{n}$, 
and $\sigma^{n}\colon U_n\to X$ is a local homeomorphism. 

\begin{definition}
For $\xi,\eta\in C_c(U)$, 
we define $\ip{\xi}{\eta}\in C_c(X)$ by 
\[
\ip{\xi}{\eta}(x)=\sum_{\sigma(y)=x}\overline{\xi}(y)\eta(y)
\]
for $x\in X$, 
where $\overline{\xi}\in C_c(U)$ 
is defined by $\overline{\xi}(y)=\overline{\xi(y)}$ for $y\in U$. 
\end{definition}

In the expression above, 
we implicitly assume that $y$ is in $U$. 
We use this convention. 
Namely if we write $\sigma^n(x)=\cdots$, 
we assume $x\in U_n$. 
The fact that $\ip{\xi}{\eta}$ is in $C_c(X)$ 
can be shown by using the fact that 
$\sigma$ is a local homeomorphism 
(see \cite[Lemma~1.5]{K1}). 

\begin{definition}\label{Def:C*S}
The \Ca $C^*(\Sigma)$ of an SGDS $\Sigma=(X,\sigma)$ 
is the universal \Ca generated by the images of 
a \shom $t^0\colon C_0(X)\to C^*(\Sigma)$ 
and a linear map $t^1\colon C_c(U)\to C^*(\Sigma)$ 
satisfying 
\begin{enumerate}
\rom
\item $t^1(\xi)^*t^1(\eta)=t^0(\ip{\xi}{\eta})$ for $\xi,\eta\in C_c(U)$, 
\item $t^1(\xi)t^1(\eta)^*=t^0(\xi\overline{\eta})$ for $\xi,\eta\in C_c(V)$ 
where $V\subset U$ is an open subset on which $\sigma$ is injective. 
\end{enumerate}
\end{definition}

We call the pair $(t^0,t^1)$ in Definition~\ref{Def:C*S} 
the {\em universal} pair for $C^*(\Sigma)$. 
Hereafter, 
we investigate properties of the universal pair $(t^0,t^1)$ for $C^*(\Sigma)$. 
For $f \in C_0(X)$ and $\xi \in C_c(U)$, 
we have $t^1(\xi)t^0(f)=t^1(\xi (f\circ \sigma))$ 
because one can see $d^*d=0$ for $d= t^1(\xi)t^0(f)-t^1(\xi (f\circ \sigma))$ 
by the computation 
\begin{align*}
t^1(\eta)^*\big(t^1(\xi)t^0(f)-t^1(\xi (f\circ \sigma))\big)
&=t^0(\ip{\eta}{\xi})t^0(f)-t^0(\ip{\eta}{\xi (f\circ \sigma)})\\
&=t^0\big(\ip{\eta}{\xi}f-\ip{\eta}{\xi (f\circ \sigma)})=0
\end{align*}
for arbitrary $\eta \in C_c(U)$. 
For $f \in C_0(X)$ and $\xi \in C_c(U)$, 
we have $t^0(f)t^1(\xi)=t^1(f \xi)$. 
To see this, first we may assume $\xi \in C_c(V)$ 
where $V\subset U$ is an open subset on which $\sigma$ is injective 
by the partition of unity 
because the fact that $\sigma$ is locally homeomorphic implies 
that the set of such $V$ covers $U$. 
Now for arbitrary $\eta \in C_c(V)$, we have 
\begin{align*}
\big(t^0(f)t^1(\xi) - t^1(f \xi)\big)t^1(\eta)^*
&= t^0(f)t^0(\xi\overline{\eta}) - t^0(f\xi\overline{\eta}) 
=0. 
\end{align*}
Therefore we get $dd^*=0$ for $d= t^0(f)t^1(\xi) - t^1(f \xi)$. 
We have shown that $t^0(f)t^1(\xi)=t^1(f \xi)$ 
for all $f \in C_0(X)$ and $\xi \in C_c(U)$. 
Finally we see that for $f \in C_c(U)$ 
there exist $\xi_1,\ldots ,\xi_n, \eta_1,\ldots ,\eta_n \in C_c(U)$ 
such that 
\begin{align*}
t^0(f)=\sum_{k=1}^n t^1(\xi_k)t^1(\eta_k)^*. 
\end{align*}
By the partition of unity, 
there exist $f_k \in C_c(V_k)$ for $k=1,\ldots, n$ 
such that $f= \sum_{k=1}^n f_k$ 
where $V_k\subset U$ is an open subset on which $\sigma$ is injective 
for $k=1,\ldots, n$. 
Then for $k=1,\ldots, n$ one can find $\xi_k,\eta_k \in C_c(V_k)$ 
such that $\xi_k\overline{\eta_k}=f_k$ 
(for example one can take $\eta_k=|f_k|^{1/2}$). 
Then we have 
\begin{align*}
t^0(f)=\sum_{k=1}^n t^0(f_k)=\sum_{k=1}^n t^1(\xi_k)t^1(\eta_k)^*. 
\end{align*}
From these computation, one can see that 
the \Ca $C^*(\Sigma)$ is the \Ca of the topological graph 
$E=(X,U,\sigma,\iota)$, 
where $\iota \colon U \to X$ is the embedding, 
as defined in \cite[Definition~2.10]{K1}. 
Thus by \cite[Proposition~10.9]{K2}, 
the \Ca $C^*(\Sigma)$ is isomorphic to the \Ca 
defined in \cite{R2} when $X$ is second countable. 

We try to make this paper self-contained as well as possible. 
We quote results in \cite{K1} which is used to prove 
the main theorem (Theorem~\ref{MainThm}) 
only in Appendix~\ref{Sec:UT}. 
We know that both $t^0$ and $t^1$ are injective 
by \cite[Proposition~3.7]{K1}. 
This fact is reproved in Lemma~\ref{Lem:inj}. 

\begin{definition}
Let $n\in \N$. 
For $\xi,\eta \in C_c(U_n)$, 
we define $\ip{\xi}{\eta}_n \in C_c(X)$ by
\[
\ip{\xi}{\eta}_n(x)=\sum_{\sigma^n(y)=x}\overline{\xi}(y)\eta(y)
\]
for $x\in X$. 
\end{definition}

Note that $\ip{\xi}{\eta}_0=\overline{\xi}\eta$ 
for $\xi,\eta \in C_c(X)$, and $\ip{\xi}{\eta}_1=\ip{\xi}{\eta}$ 
for $\xi,\eta \in C_c(U)$. 

\begin{lemma}\label{Lem:ip1}
Let $n,m\in \N$. 
For $\xi_1,\eta_1 \in C_c(U_n)$ and $\xi_2,\eta_2 \in C_c(U_m)$, 
define $\xi,\eta \in C_c(U_{n+m})$ 
by $\xi(x)=\xi_1(x)\xi_2(\sigma^n(x))$ and 
$\eta(x)=\eta_1(x)\eta_2(\sigma^n(x))$ for $x \in U_{n+m}$. 
Then we have
\[
\ip{\xi_2}{\ip{\xi_1}{\eta_1}_n\eta_2}_m
=\ip{\xi}{\eta}_{n+m}
\]
\end{lemma}

\begin{proof}
For $x \in X$, we have 
\begin{align*}
\ip{\xi_2}{\ip{\xi_1}{\eta_1}_n\eta_2}_m(x)
&=\sum_{\sigma^m(y)=x}\overline{\xi_2}(y)\ip{\xi_1}{\eta_1}_n(y)\eta_2(y)\\
&=\sum_{\sigma^m(y)=x}\overline{\xi_2}(y)\Big(\sum_{\sigma^n(z)=y}\overline{\xi_1}(z)\eta_1(z)\Big)\eta_2(y)\\
&=\sum_{\sigma^m(y)=x}\sum_{\sigma^n(z)=y}
\overline{\xi_2}(\sigma^n(z))\overline{\xi_1}(z)\eta_1(z)\eta_2(\sigma^n(z))\\
&=\sum_{\sigma^{n+m}(z)=x}\overline{\xi}(z)\eta(z)\\
&=\ip{\xi}{\eta}_{n+m}(x). \qedhere
\end{align*}
\end{proof}

\begin{lemma}\label{Lem:ip2}
Let $n\in \N$. 
For $\xi_1,\ldots,\xi_n,\eta_1,\ldots,\eta_n \in C_c(U)$, 
define $\xi,\eta \in C_c(U_{n})$ 
by 
\begin{align*}
\xi(x)&=\xi_1(x)\xi_2(\sigma(x))\cdots \xi_n(\sigma^{n-1}(x))\\
\eta(x)&=\eta_1(x)\eta_2(\sigma(x))\cdots \eta_n(\sigma^{n-1}(x))
\end{align*}
for $x \in U_{n}$. 
Then we have
\[
\ip{\xi_n}{\ip{\xi_{n-1}}{\cdots ,\ip{\xi_2}{\ip{\xi_1}{\eta_1}\eta_2}
\cdots \eta_{n-1}}\eta_n}
=\ip{\xi}{\eta}_{n}
\]
\end{lemma}

\begin{proof}
Apply Lemma~\ref{Lem:ip1} $(n-1)$ times. 
\end{proof}

\begin{lemma}\label{Lem:ttt}
Let $n\in \N$. 
For $\xi_1,\ldots,\xi_n,\eta_1,\ldots,\eta_n \in C_c(U)$, 
we have
\begin{align*}
\big(t^1(\xi_1)t^1(\xi_2)\cdots t^1(\xi_n)\big)^*&
\big(t^1(\eta_1)t^1(\eta_2)\cdots t^1(\eta_n)\big)\\
&=
t^0\big(\ip{\xi_n}{\ip{\xi_{n-1}}{\cdots ,\ip{\xi_2}{\ip{\xi_1}{\eta_1}\eta_2}
\cdots \eta_{n-1}}\eta_n}\big)
\end{align*}
\end{lemma}

\begin{proof}
This follows from direct computation. 
\end{proof}

\begin{lemma}\label{Lem:fctn}
Let $n,m\in \N$. 
For $\xi \in C_c(U_{n+m})$, there exist
$\xi_1 \in C_c(U_n)$ and $\xi_2 \in C_c(U_m)$ 
such that 
$\xi(x)=\xi_1(x)\xi_2(\sigma^n(x))$ for $x \in U_{n+m}$. 
\end{lemma}

\begin{proof}
Take $\xi_1=\xi$, and choose 
$\xi_2 \in C_c(U_m)$ which is 1 on the compact set 
$\sigma^n(C) \subset U_m$ 
where $C \subset U_{n+m}$ is the compact support of $\xi$. 
\end{proof}

\begin{lemma}\label{Lem:fctn2}
Let $n\in \N$. 
For $\xi \in C_c(U_{n})$, there exist
$\xi_1,\xi_2, \ldots,\xi_n \in C_c(U)$ 
such that 
$\xi(x)=\xi_1(x)\xi_2(\sigma(x))\cdots \xi_n(\sigma^{n-1}(x))$
for $x \in U_{n}$. 
\end{lemma}

\begin{proof}
Apply Lemma~\ref{Lem:fctn} $(n-1)$ times. 
\end{proof}

For each integer $n\geq 2$, 
we define a linear map $t^n\colon C_c(U_n)\to C^*(\Sigma)$. 
Take $\xi\in C_c(U_n)$. 
Take $\xi_1,\ldots,\xi_n\in C_c(U)$ 
such that 
\[
\xi(x)=\xi_1(x)\xi_2(\sigma(x))\cdots \xi_n(\sigma^{n-1}(x))
\] 
for all $x \in U_n$. 
Such functions exist by Lemma~\ref{Lem:fctn2}. 
We set $t^n(\xi)=t^1(\xi_1)t^1(\xi_2)\cdots t^1(\xi_n)$. 
This definition is well-defined 
because Lemma~\ref{Lem:ip2} and Lemma~\ref{Lem:ttt} imply $d^*d=0$ 
for the difference $d$ of two definitions. 
By the same reason, we see that $t^n$ is linear. 
We also see the following. 

\begin{lemma}\label{Lem:t*t}
For $\xi,\eta \in C_c(U_n)$, 
we have $t^n(\xi)^*t^n(\eta)=t^0(\ip{\xi}{\eta}_n)$. 
\end{lemma}

\begin{proof}
This follows from Lemma~\ref{Lem:ip2} and Lemma~\ref{Lem:ttt}. 
\end{proof}

\begin{lemma}\label{Lem:tt=t}
Let $n,m\in \N$. 
For $\xi_1 \in C_c(U_n)$ and $\xi_2 \in C_c(U_m)$, 
define $\xi \in C_c(U_{n+m})$ 
by $\xi(x)=\xi_1(x)\xi_2(\sigma^n(x))$ for $x \in U_{n+m}$. 
Then we have
\[
t^n(\xi_1)t^m(\xi_2)=t^{n+m}(\xi)
\]
\end{lemma}

\begin{proof}
When $n \geq 1$ and $m \geq 1$, 
this follows from the definition of $t^n$. 
When $n =0$ or $m=0$, 
this follows from a similar computation as 
the one after Definition~\ref{Def:C*S}. 
\end{proof}

\begin{lemma}\label{Lem:t^*t=t}
Let $n,m\in \N$. 
For $\xi \in C_c(U_n)$ and $\eta \in C_c(U_{n+m}) \subset C_c(U_n)$, 
we have 
\[
t^n(\xi)^*t^{n+m}(\eta)
=t^m(\ip{\xi}{\eta}_n).
\]
\end{lemma}

\begin{proof}
By Lemma~\ref{Lem:fctn}, 
choose $\eta_1 \in C_c(U_n)$ and $\eta_2 \in C_c(U_m)$ 
such that 
$\eta(x)=\eta_1(x)\eta_2(\sigma^n(x))$ for $x \in U_{n+m}$. 
By Lemma~\ref{Lem:tt=t} and Lemma~\ref{Lem:t*t}, we have 
\begin{align*}
t^n(\xi)^*t^{n+m}(\eta)
&=t^n(\xi)^*t^{n}(\eta_1)t^{m}(\eta_2)
=t^0(\ip{\xi}{\eta_1}_n)t^{m}(\eta_2)\\
&=t^m(\ip{\xi}{\eta_1}_n\eta_2) 
=t^m(\ip{\xi}{\eta_1(\eta_2\circ \sigma^n)}_n)=t^m(\ip{\xi}{\eta}_n). 
\qedhere
\end{align*}
\end{proof}

\begin{lemma}\label{Lem:cspa}
The \Ca $C^*(\Sigma)$ is the closure of the linear span of the set 
\[
\big\{t^n(\xi)t^m(\eta)^*\mid n,m\in\N, \xi\in C_c(U_n), \eta\in C_c(U_m)\}. 
\]
\end{lemma}

\begin{proof}
By Lemma~\ref{Lem:tt=t} and Lemma~\ref{Lem:t^*t=t}, 
the set above is closed under multiplication. 
This set contains the images of $t^0$ and $t^1$. 
Hence the closure of the linear span of the set 
is $C^*(\Sigma)$. 
\end{proof}

\section{Irreducible representations of $C^*(\Sigma)$}\label{Sec:irrep}

Take an SGDS $\Sigma=(X,\sigma)$. 
In this section, we construct irreducible representations $\pi_{x_0,\gamma}$ 
of the \Ca $C^*(\Sigma)$. 

\begin{definition}
For $x_0\in X$, we define the {\em orbit} of $x_0$ by 
\[
\Orb(x_0)=\{x\in X\mid \text{$\sigma^n(x)=\sigma^m(x_0)$ 
for some $n,m\in\N$}\}.
\]
\end{definition}

Note that for $x \in \Orb(x_0)$ we have $\Orb(x)=\Orb(x_0)$. 
Hence
two orbits $\Orb(x)$ and $\Orb(y)$ are either same or disjoint. 

\begin{definition}
For $x_0\in X$, 
let $H_{x_0}$ be the Hilbert space whose complete orthonormal 
system is given by $\{\delta_x\}_{x\in \Orb(x_0)}$. 
The inner product of $H_{x_0}$ 
is denoted by $\ip{\cdot}{\cdot}_{x_0}$ 
which is linear in the second variable. 
\end{definition}

\begin{definition}
For $(x_0,\gamma_0)\in X\times\T$, 
we define two maps 
$t^0_{(x_0,\gamma_0)}\colon C_0(X)\to B(H_{x_0})$ 
and $t^1_{(x_0,\gamma_0)}\colon C_c(U)\to B(H_{x_0})$ by 
\begin{align*}
t^0_{(x_0,\gamma_0)}(f)\delta_x&=f(x)\delta_x,& 
t^1_{(x_0,\gamma_0)}(\xi)\delta_x&=\gamma_0\sum_{\sigma(y)=x}\xi(y)\delta_y, 
\end{align*}
for $x\in \Orb(x_0)$. 
\end{definition}

Take $(x_0,\gamma_0)\in X\times\T$ and fix it for a while. 
It is routine to check that $t^0_{(x_0,\gamma_0)}$ is a well-defined \shom 
and $t^1_{(x_0,\gamma_0)}$ is a well-defined linear map. 

\begin{lemma}\label{Lem:adjoint}
For $\xi\in C_c(U)$ and $y\in \Orb(x_0)$, 
we have 
\[
t^1_{(x_0,\gamma_0)}(\xi)^*\delta_y
=\begin{cases}
\gamma_0^{-1}\overline{\xi}(y)\delta_{\sigma(y)}&\text{if $y\in U$}\\
0&\text{otherwise.}
\end{cases} 
\]
\end{lemma}

\begin{proof}
For $x,y\in \Orb(x_0)$, 
we have 
\begin{align*}
\ip{\delta_x}{t^1_{(x_0,\gamma_0)}(\xi)^*\delta_y}_{x_0}
&=\ip{t^1_{(x_0,\gamma_0)}(\xi)\delta_x}{\delta_y}_{x_0}\\
&=\bigg\langle 
\gamma_0\sum_{\sigma(y')=x}\xi(y')\delta_{y'},\delta_y\bigg\rangle_{x_0}\\
&=\sum_{\sigma(y')=x}\gamma_0^{-1}\overline{\xi(y')}
\ip{\delta_{y'}}{\delta_y}_{x_0}\\
&=\begin{cases}
\gamma_0^{-1}\overline{\xi}(y)& \text{if $x=\sigma(y)$,}\\
0 & \text{otherwise}.
\end{cases}
\end{align*}
This completes the proof. 
\end{proof}

\begin{proposition}
The pair 
$(t^0_{(x_0,\gamma_0)}, t^1_{(x_0,\gamma_0)})$ 
satisfies (i) and (ii) in Definition~\ref{Def:C*S}. 
\end{proposition}

\begin{proof}
Take $\xi,\eta \in C_c(U)$. 
For $x \in \Orb(x_0)$, 
we have
\begin{align*}
t^1_{(x_0,\gamma_0)}(\xi)^*t^1_{(x_0,\gamma_0)}(\eta)\delta_x
&=t^1_{(x_0,\gamma_0)}(\xi)^*
\Big(\gamma_0\sum_{\sigma(y)=x}\eta(y)\delta_y\Big)\\
&=\gamma_0\sum_{\sigma(y)=x}\eta(y)t^1_{(x_0,\gamma_0)}(\xi)^*\delta_y\\
&=\gamma_0\sum_{\sigma(y)=x}\eta(y)
\gamma_0^{-1}\overline{\xi}(y)\delta_{\sigma(y)}\\
&=\sum_{\sigma(y)=x}\overline{\xi}(y)\eta(y)\delta_{x}\\
&=\ip{\xi}{\eta}(x)\delta_{x}
\end{align*}
by Lemma \ref{Lem:adjoint}.
This shows $t^1_{(x_0,\gamma_0)}(\xi)^*t^1_{(x_0,\gamma_0)}(\eta) =
t^0_{(x_0,\gamma_0)}(\ip{\xi}{\eta})$. 
Now take an open subset $V\subset U$ on which $\sigma$ is injective, 
and take $\xi,\eta\in C_c(V)$. 
For $y \in \Orb(x_0) \cap V$, 
we have
\begin{align*}
t^1_{(x_0,\gamma_0)}(\xi)t^1_{(x_0,\gamma_0)}(\eta)^*\delta_y
&=t^1_{(x_0,\gamma_0)}(\xi)\big(
\gamma_0^{-1}\overline{\eta}(y)\delta_{\sigma(y)}\big)\\
&=\gamma_0^{-1}\overline{\eta}(y)
\gamma_0\sum_{\sigma(y')=\sigma(y)}\xi(y')\delta_{y'}\\
&=\overline{\eta}(y)\xi(y)\delta_{y}
\end{align*}
by Lemma \ref{Lem:adjoint}.
For $y \in \Orb(x_0)\setminus V$, 
the same equation holds because the both sides become $0$. 
Hence we get $t^1_{(x_0,\gamma_0)}(\xi)t^1_{(x_0,\gamma_0)}(\eta)^*
=t^0_{(x_0,\gamma_0)}(\xi\overline{\eta})$. 
\end{proof}

\begin{definition}
We denote by 
$\pi_{(x_0,\gamma_0)}\colon C^*(\Sigma)\to B(H_{x_0})$ 
the \shom induced by the pair $(t^0_{(x_0,\gamma_0)}, t^1_{(x_0,\gamma_0)})$. 
\end{definition}

\begin{definition}
For $(x_0,\gamma_0)\in X\times\T$, 
we set $P_{(x_0,\gamma_0)}=\ker\pi_{(x_0,\gamma_0)}$. 
\end{definition}

\begin{lemma}\label{Lem:inj}
Let $(t^0,t^1)$ be the universal pair for $C^*(\Sigma)$. 
Then both $t^0$ and $t^1$ are injective. 
\end{lemma}

\begin{proof}
For each $x \in X$, 
we have 
$\ker t^0_{(x,1)} = C_0(X\setminus \overline{\Orb(x)})$. 
Since $t^0_{(x,1)}=\pi_{(x,1)} \circ t^0$, 
we have $\ker t^0 \subset \ker t^0_{(x,1)}$. 
Hence we have
\begin{align*}
\ker t^0 \subset \bigcap_{x \in X}\ker t^0_{(x,1)}
=\bigcap_{x \in X} C_0(X\setminus \overline{\Orb(x)})=0. 
\end{align*}
This shows that $t^0$ is injective. 
Take $\xi \in C_c(U)$ with $t^1(\xi)=0$. 
Then we have $t^0(\ip{\xi}{\xi})=t^1(\xi)^*t^1(\xi)=0$. 
Since $t^0$ is injective, we have $\ip{\xi}{\xi}=0$. 
This shows $\xi = 0$. 
Therefore $t^1$ is injective. 
\end{proof}

We are going to see that the representation $\pi_{(x_0,\gamma_0)}$ 
is irreducible, and hence $P_{(x_0,\gamma_0)}$ 
is a primitive ideal of $C^*(\Sigma)$. 

\begin{lemma}\label{compute}
For $n\in\N$, $\xi\in C_0(U_n)$
and $x\in \Orb(x_0)$, 
we have 
\begin{align*}
\pi_{(x_0,\gamma_0)}(t^n(\xi))\delta_x
&=\gamma_0^n\sum_{\sigma^n(y)=x}\xi(y)\delta_{y}\\ 
\pi_{(x_0,\gamma_0)}(t^n(\xi))^*\delta_x
&=\begin{cases}
\gamma_0^{-n}\overline{\xi}(x)\delta_{\sigma^n(x)}& 
\text{if $x\in U_n$,}\\
0 & \text{otherwise.}
\end{cases}
\end{align*}
\end{lemma}

\begin{proof}
By Lemma~\ref{Lem:fctn2}, 
we can choose $\xi_1,\ldots,\xi_n\in C_c(U)$ 
such that 
\[
\xi(x)=\xi_1(x)\xi_2(\sigma(x))\cdots \xi_n(\sigma^{n-1}(x))
\] 
for all $x \in U_n$. 
Then we have 
\begin{align*}
\pi_{(x_0,\gamma_0)}(t^n(\xi))\delta_x
&=\pi_{(x_0,\gamma_0)}\big(t^1(\xi_1)t^1(\xi_2)\ldots 
t^1(\xi_{n-1})t^1(\xi_n)\big)\delta_x\\
&=\pi_{(x_0,\gamma_0)}\big(t^1(\xi_1)t^1(\xi_2)\ldots t^1(\xi_{n-1})\big)
\gamma_0\sum_{\sigma(y_1)=x}\xi_n(y_1)\delta_{y_1}\\
&=\pi_{(x_0,\gamma_0)}\big(t^1(\xi_1)\ldots t^1(\xi_{n-2})\big)
\gamma_0^2\sum_{\sigma(y_1)=x}\xi_n(y_1)
\sum_{\sigma(y_2)=y_1}\xi_{n-1}(y_2)\delta_{y_2}\\
&=\pi_{(x_0,\gamma_0)}\big(t^1(\xi_1)\ldots t^1(\xi_{n-2})\big)
\gamma_0^2\sum_{\sigma^2(y_2)=x}\xi_{n-1}(y_2)\xi_n(\sigma(y_2))\delta_{y_2}\\
&\ \,\vdots\\
&=\pi_{(x_0,\gamma_0)}\big(t^1(\xi_1)\big)\gamma_0^{n-1}
\!\!\!\sum_{\sigma^{n-1}(y_{n-1})=x}\!\!\!
\xi_2(y_{n-1})\xi_3(\sigma(y_{n-1}))\cdots\\
&\phantom{\pi_{(x_0,\gamma_0)}\big(t^1(\xi_1)\big)\gamma_0^{n-1}
\!\!\!\sum_{\sigma^{n-1}(y_{n-1})=x}}
\cdots\xi_{n-1}(\sigma^{n-3}(y_{n-1}))\xi_n(\sigma^{n-2}(y_{n-1}))\delta_{y_{n-1}}\\
&=\gamma_0^n\sum_{\sigma^n(y)=x}\xi(y)\delta_{y}. 
\end{align*}
From this equation, 
we can compute $\pi_{(x_0,\gamma_0)}(t^n(\xi))^*\delta_x$ 
in a similar way to the proof of Lemma \ref{Lem:adjoint}. 
\end{proof}

\begin{proposition}
For $(x_0,\gamma_0)\in X\times\T$, 
the representation 
$\pi_{(x_0,\gamma_0)}\colon C^*(\Sigma)\to B(H_{x_0})$ 
is irreducible. 
\end{proposition}

\begin{proof}
Let $\{e_{x,y}\}_{x,y \in \Orb(x_0)}$ be 
the matrix units of $B(H_{x_0})$. 
Namely $e_{x,y} \in B(H_{x_0})$ satisfies 
\[
e_{x,y}\delta_z = \begin{cases}
\delta_x& \text{if $z=y$,}\\
0 & \text{otherwise}
\end{cases}
\]
for $z \in \Orb(x_0)$. 
For all $x\in \Orb(x_0)$, 
it is standard to see that 
$e_{x,x}$ is in the weak closure of 
$\pi_{(x_0,\gamma_0)}(t^0(C_0(X)))=t^0_{(x_0,\gamma_0)}(C_0(X))
\subset B(H_{x_0})$. 
Take $x \in \Orb(x_0)$ with $x \in U$. 
Take $\xi \in C_0(U)$ with $\xi(x)=\gamma_0^{-1}$. 
Then we have $e_{x,x}t_{(x_0,\gamma_0)}^1(\xi)=e_{x,\sigma(x)}$. 
Hence $e_{x,\sigma(x)}$ is in the weak closure of 
$\pi_{(x_0,\gamma_0)}(C^*(\Sigma))$ 
for all $x\in \Orb(x_0)$ with $x \in U$. 
Take $x,y \in \Orb(x_0)$. 
Then there exist $n,m \in \N$ with $\sigma^n(x)=\sigma^m(y)$. 
We have that 
\begin{align*}
e_{x,y} &= e_{x,\sigma^n(x)}(e_{y,\sigma^m(y)})^* \\
&= e_{x,x}e_{x,\sigma(x)}e_{\sigma(x),\sigma^2(x)}\ldots
e_{\sigma^{n-1}(x),\sigma^n(x)}
\big(e_{y,y}e_{y,\sigma(y)}e_{\sigma(y),\sigma^2(y)}\ldots
e_{\sigma^{m-1}(y),\sigma^m(y)}\big)^*
\end{align*}
is in the weak closure of 
$\pi_{(x_0,\gamma_0)}(C^*(\Sigma))$. 
Hence the weak closure of 
$\pi_{(x_0,\gamma_0)}(C^*(\Sigma))$. 
is whole $B(H_{x_0})$. 
\end{proof}

\begin{corollary}
For $(x_0,\gamma_0)\in X\times\T$, 
the ideal $P_{(x_0,\gamma_0)}$ is primitive. 
\end{corollary}

\begin{lemma}\label{inv0}
For $x\in U$ and $\gamma\in\T$, 
we have $P_{(x,\gamma)}=P_{(\sigma(x),\gamma)}$. 
\end{lemma}

\begin{proof}
Since $\Orb(x)=\Orb(\sigma(x))$, 
we have $\pi_{(x,\gamma)}=\pi_{(\sigma(x),\gamma)}$. 
Hence $P_{(x,\gamma)}=P_{(\sigma(x),\gamma)}$. 
\end{proof}

\begin{definition}
Let $x\in X$. 
If there exist $k,n\in\N$ with $n \geq 1$ 
such that $\sigma^{k+n}(x)=\sigma^{k}(x)$, 
then we say that $x$ is {\em periodic}, 
and define its {\em period} $p(x)$ 
to be the smallest positive integer $n$ 
satisfying $\sigma^{k+n}(x)=\sigma^{k}(x)$ for some $k$. 
We also denote $l(x)$ be the smallest natural number $k$ 
satisfying $\sigma^{k+p(x)}(x)=\sigma^{k}(x)$. 
A point $x$ which is not periodic 
is said to be {\em aperiodic}. 
We set $p(x)=l(x)=\infty$ for an aperiodic point $x$. 
\end{definition}

It is fairly easy to see, but worth remarking, that 
we have $\sigma^k(x)=\sigma^l(x)$ for $k,l\in \N$ with $k>l$ 
if and only if $l\geq l(x)$ and $k-l\in p(x)\N$.

\begin{lemma}\label{aper}
For an aperiodic point $x_0\in X$, 
we have $P_{(x_0,\gamma)}=P_{(x_0,1)}$ 
for all $\gamma\in\T$. 
\end{lemma}

\begin{proof}
If $x_0\in X$ is aperiodic, 
we can define a map $c\colon \Orb(x_0)\to \Z$ 
such that $c(x_0)=0$ and $c(\sigma(x))=c(x)-1$ for $x\in U$. 
We define a unitary $u_\gamma\in B(H_{x_0})$ 
by $u_\gamma\delta_x=\gamma^{c(x)}\delta_x$ for $x\in \Orb(x_0)$. 
It is not difficult to check that 
\begin{align*}
u_\gamma t^0_{(x_0,1)}(f)u_\gamma^*
&=t^0_{(x_0,1)}(f)=t^0_{(x_0,\gamma)}(f),\\ 
u_\gamma t^1_{(x_0,1)}(\xi)u_\gamma^*
&=\gamma t^1_{(x_0,1)}(\xi)=t^1_{(x_0,\gamma)}(\xi)
\end{align*}
for $f\in C_0(X)$ and $\xi\in C_c(U)$. 
Hence two representation $\pi_{(x_0,\gamma)}$ 
and $\pi_{(x_0,1)}$ are unitarily equivalent. 
This shows $P_{(x_0,\gamma)}=P_{(x_0,1)}$. 
\end{proof}

We denote the elements of $\Z/n\Z$ by $\{0,1,\ldots,n-1\}$, 
and sometimes consider them as elements in $\Z$. 

\begin{lemma}\label{per1}
For a periodic point $x_0\in X$ with period $n$, 
we have $P_{(x_0,\gamma)}=P_{(x_0,\mu)}$ 
if $\gamma^n=\mu^n$. 
\end{lemma}

\begin{proof}
Similarly as in the proof of Lemma \ref{aper}, 
we can define a map $c\colon \Orb(x_0)\to \Z/n\Z$ 
such that $c(x_0)=0$ and $c(\sigma(x))=c(x)-1$ for $x\in U$. 
Set $\lambda :=\gamma\overline{\mu}\in\T$. 
We have $\lambda^n=1$. 
Hence we can define $u_\lambda\in B(H_{x_0})$ 
by $u_\lambda\delta_x=\lambda^{c(x)}\delta_x$ for $x\in \Orb(x_0)$. 
Similarly as in the proof of Lemma \ref{aper}, 
two representations $\pi_{(x_0,\gamma)}$ 
and $\pi_{(x_0,\mu)}$ are unitarily equivalent 
by the unitary $u_\lambda$. 
Hence we obtain $P_{(x_0,\gamma)}=P_{(x_0,\mu)}$. 
\end{proof}

\begin{definition}
For an integer $n\geq 2$, 
we denote by $\zeta_n=e^{2\pi i/n}$ 
the $n$-th root of unity. 
\end{definition}

\begin{lemma}\label{Lem:gnmn}
Let $x_0\in X$ be a periodic point with period $n$. 
If $x_0$ is isolated in $\Orb(x_0)$, 
then for $\gamma,\mu\in\T$, 
$P_{(x_0,\gamma)}=P_{(x_0,\mu)}$ 
if and only if $\gamma^n=\mu^n$. 
\end{lemma}

\begin{proof}
Let $\Lambda=\{\sigma^{l(x_0)}(x_0),\ldots,\sigma^{l(x_0)+n-1}(x_0)\}$. 
Since $\sigma$ is a local homeomorphism, 
if $x_0$ is isolated in $\Orb(x_0)$ 
we can find an open subset $V$ of $X$ 
such that $V\cap \Orb(x_0)=\Lambda$. 
Hence there exists $\xi\in C_c(U)$ 
such that $\xi(x)=1$ for $x\in\Lambda$ 
and $\xi(x)=0$ for $x\in\Orb(x_0)\setminus \Lambda$. 
For each $\gamma\in\T$, 
non-zero elements of the spectrum of $t^1_{(x_0,\gamma)}(\xi)\in B(H_{x_0})$ 
are $\gamma,\gamma\zeta_n,\ldots,\gamma\zeta_n^{n-1}$. 
Hence those of the image of $t^1(\xi)$ 
via the natural surjection $C^*(\Sigma)\to C^*(\Sigma)/P_{(x_0,\gamma)}$ 
are also $\gamma,\gamma\zeta_n,\ldots,\gamma\zeta_n^{n-1}$. 
This shows that 
if $P_{(x_0,\gamma)}=P_{(x_0,\mu)}$ then $\gamma^n=\mu^n$. 
The converse follows from Lemma \ref{per1}. 
\end{proof}

\begin{remark}
We will see in Proposition~\ref{Prop:notiso}
that for a periodic point $x_0\in X$ 
such that $x_0$ is not isolated in $\Orb(x_0)$, 
we have $P_{(x_0,\gamma)}=P_{(x_0,1)}$ for all $\gamma\in\T$. 
\end{remark}

\begin{remark}
We will see in Remark~\ref{Rem:2ndctbl}
that if $X$ is second countable $P_{(x,\gamma)}$'s are 
whole primitive ideal. 
(See \cite{SW} for the case $U=X$.) 
In general, there is a primitive ideal 
which is not in the form $P_{(x,\gamma)}$
(see \cite[Example~13.2]{K3}). 
\end{remark}

It is probably impossible to list all primitive ideals 
in terms of elements of $X$ and $\T$. 
However it is possible to list all prime ideals 
in terms of subsets of $X$ and $\T$. 
From next section, 
we will do this. 
This will show that the set $\{P_{(x,\gamma)}\}$ is sufficiently 
large (Corollary~\ref{Cor:inter}).

\section{gauge invariant ideals and $\sigma$-invariant sets}\label{Sec:gii}

\begin{definition}
A subset $X'$ of $X$ is said to be {\em $\sigma$-invariant} 
when $x\in X'$ if and only if $\sigma(x)\in X'$ for all $x\in U$. 
\end{definition}

An ideal $I$ of $C^*(\Sigma)$ is said to be {\em gauge-invariant} 
if $\beta_z(I)=I$ for all $z \in \T$ where $\beta$ is the gauge action 
defined in Appendix~\ref{Sec:UT}. 
We will see in Proposition~\ref{Prop:bij} 
that the set of gauge-invariant ideals 
corresponds bijectively to the set of closed $\sigma$-invariant subsets 
of $X$. 

\begin{definition}
For an ideal $I$ of $C^*(\Sigma)$, we define a closed set $X_I$ 
of $X$ by $t^0(C_0(X\setminus X_I))=t^0(C_0(X))\cap I$. 
\end{definition}

The following lemma is easy to see from the definition. 

\begin{lemma}\label{Lem:Xeasy}
For two ideals $I_1,I_2$ of $C^*(\Sigma)$, 
we have $X_{I_1\cap I_2}=X_{I_1} \cup X_{I_2}$. 
If $I_1 \subset I_2$, then $X_{I_1} \supset X_{I_2}$. 
\end{lemma}

\begin{proposition}
For an ideal $I$ of $C^*(\Sigma)$, the closed set $X_I$ 
is $\sigma$-invariant. 
\end{proposition}

\begin{proof}
Take $x\in U$. 
Suppose $x \notin X_I$. 
Choose an open subset $V$ such that 
$x \in V \subset U$, $V \cap X_I=\emptyset$ 
and $\sigma$ is injective on $V$. 
Take $\xi \in C_c(V)$ with $\xi(x)=1$. 
Then $\xi\overline{\xi} \in C_0(X \setminus X_I)$. 
Hence $t^1(\xi)t^1(\xi)^*=t^0(\xi\overline{\xi}) \in I$. 
This implies $t^0(\ip{\xi}{\xi})=t^1(\xi)^*t^1(\xi) \in I$. 
Hence $\ip{\xi}{\xi} \in C_0(X \setminus X_I)$. 
Since $\ip{\xi}{\xi}(\sigma(x))=|\xi(x)|^2=1$, 
we have $\sigma(x) \notin X_I$. 
Now suppose $\sigma(x) \notin X_I$. 
Choose an open subset $V$ such that 
$x \in V \subset U$, $V \cap \sigma^{-1}(X_I)=\emptyset$ 
and $\sigma$ is injective on $V$. 
Take $\xi \in C_c(V)$ with $\xi(x)=1$. 
Then $\ip{\xi}{\xi} \in C_0(X \setminus X_I)$. 
Hence $t^1(\xi)^*t^1(\xi) = t^0(\ip{\xi}{\xi}) \in I$. 
This implies $t^0(\xi\overline{\xi})=t^1(\xi)t^1(\xi)^* \in I$. 
Hence $\xi\overline{\xi} \in C_0(X \setminus X_I)$. 
Since $\xi\overline{\xi}(x)=|\xi(x)|^2=1$, 
we have $x \notin X_I$. 
Thus we have shown that 
$x\in X_I$ if and only if $\sigma(x)\in X_I$. 
\end{proof}

Take a closed $\sigma$-invariant subset $X'$ of $X$. 
We define an SGDS $\Sigma'=(X',\sigma')$ where $\sigma'$ is a 
restriction of $\sigma$ on $U':=X' \cap U$. 
Let $(t'^0,t'^1)$ be the universal pair for $C^*(\Sigma')$. 
We have the following. 

\begin{lemma}\label{Lem:StoS'}
There exists a surjection $\varPhi\colon C^*(\Sigma) \to C^*(\Sigma')$ 
such that $\varPhi(t^0(f))=t'^0(f|_{X'})$ 
and $\varPhi(t^1(\xi))=t'^1(\xi|_{U'})$. 
\end{lemma}

\begin{proof}
It suffices to see that the pair of maps 
$C_0(X)\ni f \mapsto t'^0(f|_{X'}) \in C^*(\Sigma')$ 
and $C_c(U)\ni \xi \mapsto t'^1(\xi|_{U'}) \in C^*(\Sigma')$ 
satisfy (i) and (ii) in Definition~\ref{Def:C*S} for $\Sigma$. 
This follows from the fact that $t'^0$ and $t'^1$ satisfy 
(i) and (ii) in Definition~\ref{Def:C*S} for $\Sigma'$. 
\end{proof}

\begin{definition}
For a closed $\sigma$-invariant subset $X'$ of $X$, 
let $I_{X'}$ be the kernel of the surjection 
$\varPhi\colon C^*(\Sigma) \to C^*(\Sigma')$ 
in Lemma~\ref{Lem:StoS'}. 
\end{definition}

\begin{proposition}\label{Prop:XIX}
For a closed $\sigma$-invariant subset $X'$ of $X$, 
we have $X_{I_{X'}}=X'$. 
\end{proposition}

\begin{proof}
This follows from the fact that the kernel of $\varPhi \circ t^0$ 
in Lemma~\ref{Lem:StoS'} is $C_0(X\setminus X')$. 
\end{proof}

\begin{lemma}\label{Lem:IXinv}
For a closed $\sigma$-invariant subset $X'$ of $X$, 
the ideal $I_{X'}$ is gauge-invariant. 
\end{lemma}

\begin{proof}
This follows from the fact that 
the surjection $\varPhi\colon C^*(\Sigma) \to C^*(\Sigma')$ 
in Lemma~\ref{Lem:StoS'} commutes with the gauge actions. 
\end{proof}

\begin{lemma}\label{Lem:IXICI}
For an ideal $I$ of $C^*(\Sigma)$, 
we have $I_{X_I}\subset I$. 
\end{lemma}

\begin{proof}
Takes an deal $I$ of $C^*(\Sigma)$, 
and denote the natural surjection 
by $\pi\colon C^*(\Sigma) \to C^*(\Sigma)/I$. 
We set $\Sigma_I:=(X_I,\sigma_I)$ 
where $\sigma_I$ is the restriction of $\sigma$ to $U_I:=U\cap X_I$. 
By the definition of $X_I$, 
we can define a \shom $t'^0\colon C_0(X_I) \to C^*(\Sigma)/I$ 
such that $t'^0(f|_{X_I})=\pi(t^0(f))$ for all $f \in C_0(X)$. 
We can also define a linear map $t'^1\colon C_c(U_I) \to C^*(\Sigma)/I$ 
such that $t'^1(\xi|_{U_I})=\pi(t^1(\xi))$ for all $\xi \in C_c(U)$ 
because for $\eta \in C_c(U)$ with $\eta|_{U_I}=0$ 
we have $\ip{\eta}{\eta} \in C_0(X \setminus X_I)$ 
and hence $t^1(\eta)\in I$. 
The pair $(t'^0,t'^1)$ satisfies 
(i) and (ii) in Definition~\ref{Def:C*S} for $\Sigma_I$. 
Hence we get a \shom $\varPsi\colon C^*(\Sigma_I) \to C^*(\Sigma)/I$ 
with $\pi=\varPsi\circ \varPhi_I$ 
where $\varPhi_I\colon C^*(\Sigma) \to C^*(\Sigma_I)$ 
is the surjection whose kernel is $I_{X_I}$. 
Hence we get $I_{X_I} \subset I$. 
\end{proof}

\begin{proposition}\label{Prop:IXI}
For an ideal $I$ of $C^*(\Sigma)$, 
we have $I_{X_I} = I$ if and only if 
$I$ is gauge-invariant. 
\end{proposition}

\begin{proof}
The ``only if'' part follows from Lemma~\ref{Lem:IXinv}. 
Let $I$ be a gauge-invariant ideal $I$ of $C^*(\Sigma)$. 
Then we can define an action of $\T$ on $C^*(\Sigma)/I$ 
so that the natural surjection $\pi\colon C^*(\Sigma) \to C^*(\Sigma)/I$ 
becomes equivariant. 
Then the map $\varPsi\colon C^*(\Sigma_I) \to C^*(\Sigma)/I$ 
in the proof of Lemma~\ref{Lem:IXICI} becomes also equivariant. 
The map $\varPsi$ is injective on $t^0(C_0(X_I))$ by the definition of $X_I$. 
Hence by Proposition~\ref{Prop:GIUT} $\varPsi$ is injective. 
Therefore we have $I_{X_I} = I$. 
\end{proof}

\begin{proposition}\label{Prop:bij}
Through the maps $I\mapsto X_I$ and $X'\mapsto I_{X'}$, 
the set of gauge-invariant ideals 
corresponds bijectively to the set of closed $\sigma$-invariant subsets 
of $X$. 
\end{proposition}

\begin{proof}
This follows from Proposition~\ref{Prop:XIX} and Proposition~\ref{Prop:IXI}. 
\end{proof}

Now we have the following. 

\begin{lemma}\label{Lem:Ieasy}
For closed $\sigma$-invariant subsets $X_1$ and $X_2$ of $X$, 
we have $I_{X_1\cup X_2} = I_{X_1}\cap I_{X_2}$. 
If $X_1 \subset X_2$, then $I_{X_1}\supset I_{X_2}$. 
\end{lemma}

\begin{proof}
By Lemma~\ref{Lem:Xeasy} and Proposition~\ref{Prop:XIX}, 
we have
\begin{align*}
X_{I_{X_1}\cap I_{X_2}}=X_{I_{X_1}}\cup X_{I_{X_2}}=X_1\cup X_2.
\end{align*}
Since $I_{X_1}\cap I_{X_2}$ is gauge-invariant, 
we have 
\[
I_{X_1}\cap I_{X_2}=I_{X_{I_{X_1}\cap I_{X_2}}}=I_{X_1\cup X_2}
\]
by Proposition~\ref{Prop:IXI}. 
If $X_1 \subset X_2$, then $X_2 = X_1\cup X_2$. 
Hence we have $I_{X_2}=I_{X_1\cup X_2}=I_{X_1}\cap I_{X_2}$. 
This shows $I_{X_1}\supset I_{X_2}$. 
\end{proof}

\begin{definition}
For a closed $\sigma$-invariant set $X'$, 
we say $X'$ is essentially free 
if the SGDS $\Sigma'=(X',\sigma|_{U\cap X'})$ 
is essentially free as defined in Definition~\ref{Def:esfree}. 
\end{definition}

\begin{proposition}\label{Prop:IXI2}
For an ideal $I$ of $C^*(\Sigma)$, 
if $X_I$ is essentially free then 
we have $I_{X_I} = I$. 
\end{proposition}

\begin{proof}
The proof goes similarly as in Proposition~\ref{Prop:IXI} 
using Proposition~\ref{Prop:CKUT} instead of Proposition~\ref{Prop:GIUT}. 
\end{proof}

\section{Prime ideals}\label{Sec:prime}

An ideal $I$ of a \Ca $A$ is said to be {\em prime} 
if $I \neq A$ and for two ideals $I_1,I_2$ of a \Ca $A$ 
with $I_1 \cap I_2 \subset I$ we have either $I_1 \subset I$ 
or $I_2 \subset I$. 
A primitive ideal is prime \cite[II.6.1.11]{B}. 
The converse is true when $X$ is second countable (cf.\ \cite[II.6.5.15]{B}), 
but not true in general 
(see \cite[Example~13.3]{K3} and \cite[Theorem~5.4]{Ksh}). 

\begin{definition}
A closed $\sigma$-invariant set $X$ is said to be 
{\em irreducible} if it is non-empty and 
we have either $X=X_1$ or $X=X_2$ 
for two closed $\sigma$-invariant sets $X_1,X_2$ 
satisfying $X=X_1 \cup X_2$. 
\end{definition}

\begin{proposition}\label{Prop:prir}
If an ideal $I$ of $C^*(\Sigma)$ is prime, 
then $X_I$ is irreducible. 
\end{proposition}

\begin{proof}
If $X_I = \emptyset$, then $I=A$ which contradicts 
that $I$ is prime. 
Hence $X_I \neq \emptyset$. 
Take two closed $\sigma$-invariant sets $X_1,X_2$ 
satisfying $X_I=X_1 \cup X_2$. 
Then 
\[
I_{X_1}\cap I_{X_2} = I_{X_1\cup X_2}=I_{X_I}\subset I
\]
by Lemma~\ref{Lem:Ieasy} and Lemma~\ref{Lem:IXICI}. 
Since $I$ is prime, 
we have either $I_{X_1} \subset I$ or $I_{X_2} \subset I$. 
When $I_{X_1} \subset I$, 
we have $X_I \subset X_{I_{X_1}} = X_1 \subset X_I$ 
by Lemma~\ref{Lem:Xeasy} and Proposition~\ref{Prop:XIX}, 
and hence $X_1=X_I$. 
Similarly we have $X_2=X_I$ when $I_{X_2} \subset I$. 
Thus $X_I$ is irreducible. 
\end{proof}

\begin{proposition}\label{Prop:XP}
For $(x,\gamma)\in X\times\T$, 
we have $X_{P_{(x,\gamma)}}=\overline{\Orb(x)}$. 
\end{proposition}

\begin{proof}
This follows from 
$\ker t^0_{(x,\gamma)} = C_0(X\setminus \overline{\Orb(x)})$. 
\end{proof}

\begin{proposition}
For $x \in X$, 
$\overline{\Orb(x)}$ is an irreducible 
closed $\sigma$-invariant set.
\end{proposition}

\begin{proof}
This follows from Proposition~\ref{Prop:prir}
and Proposition~\ref{Prop:XP}. 
One can also show this directly as follows. 
Since $\Orb(x)$ is $\sigma$-invariant 
and since $\sigma$ is locally homeomorphic, 
$\overline{\Orb(x)}$ is closed and $\sigma$-invariant. 
Take two closed $\sigma$-invariant sets $X_1,X_2$ 
satisfying $\overline{\Orb(x)}=X_1 \cup X_2$. 
Then either $x \in X_1$ or $x \in X_2$. 
When $x \in X_1$ we have $X_1=\overline{\Orb(x)}$, 
whereas when $x \in X_2$ we have $X_2=\overline{\Orb(x)}$. 
Thus $\overline{\Orb(x)}$ is irreducible. 
\end{proof}

\begin{proposition}
Let $X'$ be an irreducible closed $\sigma$-invariant set. 
If $X'$ is essentially free, 
then $I_{X'}$ is prime. 
\end{proposition}

\begin{proof}
Since $X' \neq \emptyset$, we have $I_{X'} \neq A$. 
Take ideals $I_1,I_2$ of $C^*(\Sigma)$ 
such that $I_1\cap I_2 \subset I_{X'}$. 
Then we have 
\[
X_{I_1+I_{X'}} \cup X_{I_2+I_{X'}} =
X_{(I_1+I_{X'})\cap (I_2+I_{X'})} = X_{I_{X'}}=X'
\]
by Lemma~\ref{Lem:Xeasy} and Proposition~\ref{Prop:XIX}. 
Since $X'$ is irreducible, either we have $X_{I_1+I_{X'}}=X'$ 
or $X_{I_2+I_{X'}}=X'$. 
When $X_{I_1+I_{X'}}=X'$ we have $I_1+I_{X'}=I_{X'}$ 
by Proposition~\ref{Prop:IXI2} 
because $X'$ is essentially free. 
This shows $I_1 \subset I_{X'}$. 
Similarly, when $X_{I_2+I_{X'}}=X'$ we have $I_2 \subset I_{X'}$. 
Thus $I_{X'}$ is prime. 
\end{proof}

From now on, we investigate 
which irreducible closed $\sigma$-invariant set $X'$ 
becomes essentially free. 

\begin{definition}
We set 
\[
\Per(\Sigma) := \{x \in X\mid l(x)=0, \ 
\text{$x$ is isolated in $\Orb(x)$}\}. 
\]
\end{definition}

\begin{definition}
We denote by $\cA(\Sigma)$ 
the set of all irreducible closed $\sigma$-invariant 
set which is not in the form $\overline{\Orb(x)}$ for $x \in \Per(\Sigma)$. 
\end{definition}

\begin{lemma}\label{Lem:notef}
For $x \in \Per(\Sigma)$, 
an irreducible closed $\sigma$-invariant $\overline{\Orb(x)}$ 
is not essentially free. 
\end{lemma}

\begin{proof}
In the SGDS $\Sigma'=(X',\sigma|_{U\cap X'})$ where $X'=\overline{\Orb(x)}$, 
$\{x\}$ is an open subset and $l(x)=0$. 
Hence $\Sigma'$ is not essentially free. 
\end{proof}

\begin{proposition}[{cf.\ \cite[Proposition~11.3]{K3}}]\label{Prop:efA}
For an irreducible closed $\sigma$-invariant set $X'$, 
$X'$ is essentially free
if and only if $X' \in \cA(\Sigma)$ 
\end{proposition}

\begin{proof}
Take an irreducible closed $\sigma$-invariant set $X'$. 
By Lemma~\ref{Lem:notef}, 
if $X'$ is essentially free
then $X' \in \cA(\Sigma)$. 
Suppose $X'$ is not essentially free. 
Then there exists a non-empty open subset $V'$ of $X'$ 
such that $l(x)=0$ for all $x \in V'$. 
By Baire's category theorem (see, for example, \cite[Proposition 2.2]{T1}), 
there exist a non-empty open subset $V \subset V'$ of $X'$ 
and $n \in \N$ such that $p(x)=n$ for all $x \in V$. 
Take $x \in V$ arbitrarily. 
We will show that $V=\{x,\sigma(x),\ldots,\sigma^{n-1}(x)\}$. 
To the contrary, suppose there exists 
$y \in V\setminus \{x,\sigma(x),\ldots,\sigma^{n-1}(x)\}$. 
Choose open subsets $W_1,W_2$ of $V$ such that $x \in W_1$, $y \in W_2$ 
and $W_2\cap \bigcup_{k=0}^{n-1}\sigma^k(W_1)=\emptyset$. 
Then $\bigcup_{k=0}^{\infty}(\sigma^{k})^{-1}(W_1)$ is an open 
$\sigma$-invariant subset of $X'$ 
because for $l=1,2,\ldots$ we have 
$\sigma^l(W_1)\subset (\sigma^{k})^{-1}(W_1)$ 
for $k \in \N$ with $k+l \in n\N$. 
Hence $X_1:=X'\setminus \bigcup_{k=0}^{\infty}(\sigma^{k})^{-1}(W_1)$ 
is a closed $\sigma$-invariant subset of $X'$ with $X_1 \neq X'$. 
Similarly, $X_2:=X'\setminus \bigcup_{k=0}^{\infty}(\sigma^{k})^{-1}(W_2)$ 
is a closed $\sigma$-invariant subset of $X'$ with $X_2 \neq X'$. 
Since $X'$ is irreducible, we have $X_1 \cup X_2 \neq X'$. 
Hence we have 
$(\sigma^{k_0})^{-1}(W_1) \cap (\sigma^{l_0})^{-1}(W_2) \neq \emptyset$
for some $k_0,l_0 \in \N$. 
Thus we get $z \in X'$ such that $\sigma^{k_0}(z)\in W_1$ 
and $\sigma^{l_0}(z) \in W_2$. 
Then for $m \in \N$ with $l_0+mn\geq k_0$, $\sigma^{l_0+mn}(z)$ is 
in $W_2\cap \bigcup_{k=0}^{n-1}\sigma^k(W_1)$ which is empty. 
This is a contradiction. 
Hence we have shown $V=\{x,\sigma(x),\ldots,\sigma^{n-1}(x)\}$. 
Since $V$ is open in $X'$, $\{x\}$ is open in $X'$. 
Hence $x$ is isolated in $\Orb(x) \subset X'$. 
This shows $x \in \Per(\Sigma)$. 
Finally we show that $X' = \overline{\Orb(x)}$. 
To do so, take $y \in X'$. 
Take a neighborhood $W$ of $y$ arbitrarily. 
Then 
\[
X_1 := X'\setminus \bigcup_{k=0}^{\infty}(\sigma^{k})^{-1}
\Big(\bigcup_{l=0}^{\infty}\sigma^{l}(W)\Big)
\]
is a closed $\sigma$-invariant subset of $X'$ with $X_1 \neq X'$. 
Since $\{x\}$ is open in $X'$, 
$X_2=X'\setminus \Orb(x)$ is also 
a closed $\sigma$-invariant subset of $X'$ with $X_2 \neq X'$. 
Since $X'$ is irreducible, we have $X_1 \cup X_2 \neq X'$. 
Hence we have 
\[
\Orb(x) \cap \bigcup_{k=0}^{\infty}(\sigma^{k})^{-1}
\Big(\bigcup_{l=0}^{\infty}\sigma^{l}(W)\Big)
\neq \emptyset.
\]
Therefore there exists $x' \in \Orb(x)$ and $k_0,l_0\in \N$ 
such that $\sigma^{k_0}(x') \in \sigma^{l_0}(W)$. 
Hence $W \cap \Orb(x) \neq \emptyset$. 
Since $W$ is arbitrary, we have $y \in \overline{\Orb(x)}$. 
Hence $X'=\overline{\Orb(x)}$. 
This shows $X' \notin \cA(\Sigma)$. 
\end{proof}

\begin{proposition}\label{Prop:notiso}
For a periodic point $x_0\in X$ 
such that $x_0$ is not isolated in $\Orb(x_0)$, 
we have $P_{(x_0,\gamma)}=P_{(x_0,1)}$ for all $\gamma\in\T$. 
\end{proposition}

\begin{proof}
If $x_0$ is not isolated in $\Orb(x_0)$, 
then $x$ is not isolated in $\Orb(x_0)$ for all $x \in \Orb(x_0)$ 
because $\sigma$ is locally homeomorphic. 
Hence $X':=\overline{\Orb(x_0)}$ has no isolated point. 
Thus we have $X' \in \cA(\Sigma)$. 
By Proposition~\ref{Prop:efA}, $X'$ is essentially free. 
By Proposition~\ref{Prop:IXI2} and Proposition~\ref{Prop:XP}, 
we have 
$P_{(x_0,\gamma)}=I_{X'}=P_{(x_0,1)}$ for all $\gamma\in\T$. 
\end{proof}

Take $x_0 \in \Per(\Sigma)$. 
We set $n_0 := p(x_0)$. 
We are going to show that 
prime ideals $I$ 
with $X_I=\overline{\Orb(x_0)}$ 
is parametrized as $P_{(x_0,e^{2\pi it})}$ for $0\leq t <1/n_0$. 
Let us set $X':=\overline{\Orb(x_0)}$ and 
$X'' := \overline{\Orb(x_0)}\setminus \Orb(x_0)$ 
which is a closed $\sigma$-invariant subset of $X'$. 

\begin{lemma}[{cf.\ \cite[Lemma~11.10]{K3}}]\label{Lem:XI=O}
For an ideal $I$ of $C^*(\Sigma)$, 
we have $X_I=X'$ if and only if 
$I_{X'} \subset I$ and $I_{X''} \not\subset I$. 
\end{lemma}

\begin{proof}
Take an ideal $I$ of $C^*(\Sigma)$ 
with $X_I=X'$. 
Then $I_{X'} =I_{X_I}\subset I$ 
by Lemma~\ref{Lem:IXICI}. 
If $I_{X''} \subset I$ then we have 
$X_I \subset X_{I_{X''}} =X''$ by Proposition~\ref{Prop:XIX}. 
This contradicts $X_I=X'$. 
Thus we get $I_{X''} \not\subset I$. 

Conversely, take an ideal $I$ of $C^*(\Sigma)$ 
with $I_{X'} \subset I$ and $I_{X''} \not\subset I$. 
By Lemma~\ref{Lem:Xeasy} and Proposition~\ref{Prop:XIX}, 
we have $X_I \subset X_{I_{X'}}=X'$. 
Suppose $x_0 \notin X_I$. 
Then $X_I \subset X''$ since $X_I$ is $\sigma$-invariant. 
This shows $I_{X''}\subset I_{X_I} \subset I$ 
by Lemma~\ref{Lem:Ieasy} and Lemma~\ref{Lem:IXICI}. 
This contradicts $I_{X''} \not\subset I$. 
Hence we have $x_0 \in X_I$. 
This shows $X' \subset X_I$. 
Therefore we get $X_I=X'$. 
\end{proof}

Let us set $\Sigma' = (X',\sigma|_{U\cap X'})$. 
Then we have $C^*(\Sigma')=C^*(\Sigma)/I_{X'}$ 
by the definition of $I_{X'}$. 
Let $(t^0,t^1)$ be the universal pair for $C^*(\Sigma')$, 
and define $t^n$ as in Section~\ref{Sec:C*S}. 
Let us define $f_0 \in C_c(U_{n_0})$ by 
\[
f_0(x)=\begin{cases}1 & x=x_0\\
0 & x \neq x_0 \end{cases}.
\]
We set 
$p_0=t^0(f_0)\in C^*(\Sigma')$ and 
$u_0=t^{n_0}(f_0)\in C^*(\Sigma')$. 
Note that we have $u_0^*u_0=u_0u_0^*=p_0$.

\begin{lemma}[{cf.\ \cite[Lemma~11.12]{K3}}]\label{FullHered}
The corner $p_0 C^*(\Sigma') p_0$ is 
the $C^*$-subalgebra $C^*(u_0)$ generated by $u_0$. 
\end{lemma}

\begin{proof}
A corner $p_0 C^*(\Sigma') p_0$ is a closure of a linear span of 
the set 
\[
\big\{p_0t^n(\xi)t^m(\eta)^*p_0\mid n,m\in\N, \xi\in C_c(U_n), \eta\in C_c(U_m)\}. 
\]
by Lemma~\ref{Lem:cspa}. 
For $n,m\in\N, \xi\in C_c(U_n), \eta\in C_c(U_m)$, 
the element $p_0t^n(\xi)t^m(\eta)^*p_0$ is $\alpha t^n(f_0)t^m(f_0)^*$ 
for some $\alpha \in \C$. 
We also have
\[
t^n(f_0)t^m(f_0)^*
=\begin{cases}
p_0 & \text{if $n=m$}\\
u_0^k & \text{if $n-m = n_0k$ for some $k=1,2,\ldots$}\\
(u_0^k)^* & \text{if $m-n = n_0k$ for some $k=1,2,\ldots$}\\
0 & \text{otherwise}. 
\end{cases}
\]
Hence the corner $p_0 C^*(\Sigma') p_0$ is $C^*(u_0)$. 
\end{proof}

\begin{lemma}
The corner $p_0 C^*(\Sigma') p_0=C^*(u_0)$ is full 
in the ideal $I_{X''}/I_{X'} 
\subset C^*(\Sigma)/I_{X'} = C^*(\Sigma')$. 
\end{lemma}

\begin{proof}
Let $I$ be the ideal of $C^*(\Sigma')$ generated by $p_0$. 
Then $I$ is gauge-invariant. 
Consider $X_I \subset X'$. 
Since $X_I$ does not contain $x_0$. 
we have $X_I \subset X''$. 
Therefore $I_{X''} \subset I_{X_I}=I$ 
by Lemma~\ref{Lem:Ieasy} and Proposition~\ref{Prop:IXI}. 
Since $p_0 \in I_{X''}$, we have $I \subset I_{X''}$. 
Hence we get $I =I_{X''}$. 

The ideal $I_{X''}$ of $C^*(\Sigma')$ is 
equal to $I_{X''}/I_{X'} \subset C^*(\Sigma)/I_{X'}$ 
via the identification $C^*(\Sigma)/I_{X'} = C^*(\Sigma')$. 
This completes the proof. 
\end{proof}

The following lemma should be known, 
but the author could not find it in the literature. 
A similar result on primitive ideals can be found, for example, 
in \cite[II.6.5.4]{B}

\begin{lemma}\label{Lem:primeO}
Let $A$ be a \CA , 
and $I',I''$ be ideals of $A$ with $I' \subset I''$. 
Then $I\mapsto (I\cap I'')/I'$ is a bijection 
from the set of prime ideals $I$ of $A$ 
with $I' \subset I$ and $I'' \not\subset I$ 
to the set of prime ideals of $I''/I'$. 
\end{lemma}

\begin{proof}
We first show that $I \mapsto I\cap I''$ is a bijection 
from the set of prime ideals $I$ of $A$ 
with $I'' \not\subset I$ 
to the set of prime ideals of $I''$. 
Take a prime ideal $I$ of $A$ 
with $I'' \not\subset I$. 
We show that $I\cap I''$ is a prime ideal of $I''$ 
Since $I'' \not\subset I$, 
we have $I\cap I'' \neq I''$. 
Take ideals $I_1,I_2$ of $I''$ with $I_1 \cap I_2 \subset I\cap I''$. 
Then $I_1,I_2$ are ideals of $A$ with $I_1 \cap I_2 \subset I$. 
Since $I$ is prime, either $I_1 \subset I$ or $I_2 \subset I$ holds. 
Thus either $I_1 \subset I\cap I''$ or $I_2 \subset I\cap I''$ holds. 
This shows $I\cap I''$ is a prime ideal of $I''$. 
Next take two prime ideals $I_1,I_2$ of $A$ 
with $I'' \not\subset I_1$ and $I'' \not\subset I_2$. 
Suppose $I_1\cap I'' = I_2\cap I''$ and we will show $I_1=I_2$. 
We have $I_1 \supset I_1\cap I'' = I_2\cap I''$. 
Since $I_1$ is prime and $I'' \not\subset I_1$, 
we get $I_2 \subset I_1$. 
Similarly we get $I_1 \subset I_2$. 
Thus $I_1=I_2$. 
This shows that the map $I \mapsto I\cap I''$ is injective. 
Now take a prime ideal $P$ of $I''$. 
Define $I \subset A$ by 
\[
I := \{a \in A\mid \text{$ab \in P$ for all $b \in I''$}\}. 
\]
Then $I$ is an ideal of $A$ satisfying $I \cap I''=P$. 
We will show $I$ is prime. 
Take two ideals $I_1,I_2$ of $A$ with $I_1\cap I_2 \subset I$. 
Then two ideals $I_1 \cap I'', I_2 \cap I''$ of $I''$ 
satisfy $(I_1 \cap I'')\cap (I_2 \cap I'') \subset I \cap I''=P$. 
Since $P$ is a prime ideal of $I''$, we have either 
$I_1 \cap I'' \subset P$ or $I_2 \cap I'' \subset P$. 
When $I_1 \cap I'' \subset P$, we have $I_1 \subset I$ 
by the definition of $I$. 
Similarly we have $I_2 \subset I$ when $I_2 \cap I'' \subset P$. 
Thus we have shown that $I$ is prime. 
This shows that the map $I \mapsto I\cap I''$ is surjective. 
Therefore the map $I \mapsto I\cap I''$ is bijective. 
Hence the map $I \mapsto I\cap I''$ is a bijection 
from the set of prime ideals $I$ of $A$ 
with $I' \subset I$ and $I'' \not\subset I$ 
to the set of prime ideals $P$ of $I''$ with $I' \subset P$. 

It is well-known and clear that $P \mapsto P/I'$ is 
a bijection from the set of prime ideals $P$ of $I''$ 
with $I' \subset P$  
to the set of prime ideals of $I''/I'$. 
This completes the proof. 
\end{proof}

\begin{proposition}
The prime ideals $I$ 
with $X_I=\overline{\Orb(x_0)}$ 
is parametrized as $P_{(x_0,e^{2\pi it})}$ for $0\leq t <1/n_0$. 
\end{proposition}

\begin{proof}
Since $C^*(u_0)$ is hereditary and full 
in the ideal $I_{X''}/I_{X'}$, 
the map $I \mapsto I \cap C^*(u_0)$ induces 
the bijection from the set of all ideals of $I_{X''}/I_{X'}$ 
to the one of $C^*(u_0)$ preserving the intersections and inclusion. 
Hence it induces 
the bijection between the set of prime ideals. 
Combine this with Lemma~\ref{Lem:XI=O} and Lemma~\ref{Lem:primeO}, 
we see that $I \mapsto I \cap C^*(u_0)$ induces 
the bijection from the set of prime ideals $I$ 
with $X_I=\overline{\Orb(x_0)}$ 
to the set of all prime ideals of $C^*(u_0)$. 
The prime ideals of $C^*(u_0)$ is parametrized as $Q_z$ for $z \in \T$, 
where $Q_z$ is the ideal generated by $p_0-zu_0$. 
Since we have $\pi_{(x_0,\gamma)}(u_0)=\gamma^{n_0}\pi_{(x_0,\gamma)}(p_0)$, 
we get $P_{(x_0,\gamma)} \cap C^*(u_0)=Q_{\gamma^{-n_0}}$. 
Hence we see that $P_{(x_0,e^{2\pi it})}$ for $0\leq t <1/n_0$ is the 
parametrization of the prime ideals $I$ 
with $X_I=\overline{\Orb(x_0)}$. 
\end{proof}

Note that for $x_0\in \Per(\Sigma)$ with $n_0=p(x_0)$, 
we have $P_{(x,\gamma)}=P_{(x,\mu)}$ if and only if $\gamma^{n_0}=\mu^{n_0}$ 
by Lemma~\ref{Lem:gnmn}. 

Combine the discussion above, we get the following. 

\begin{theorem}
The set of all prime ideals of $C^*(\Sigma)$ is 
\[
\{I_{X'}\mid X'\in \cA(\Sigma)\}\cup
\{P_{(x,e^{2\pi it})}\mid x\in \Per(\Sigma), 0\leq t <1/p(x)\}. 
\]
\end{theorem}

\begin{remark}\label{Rem:2ndctbl}
One can show that when $X$ is second countable 
for $X'\in \cA(\Sigma)$ there exists $x \in X$ such that $X'=\overline{\Orb(x)}$ 
(cf.\ \cite[Proposition~4.14]{K3}). 
Therefore when $X$ is second countable, 
$\{P_{(x,\gamma)}\}$ are all prime (primitive) ideals of $C^*(\Sigma)$. 
\end{remark}

\begin{proposition}
Every prime ideals of $C^*(\Sigma)$ is 
an intersection of primitive ideals in the form $P_{(x,\gamma)}$. 
\end{proposition}

\begin{proof}
Take $X'\in \cA(\Sigma)$. 
We set $I = \bigcap_{x \in X', \gamma \in \T}P_{(x,\gamma)}$. 
Then $I$ is gauge-invariant. 
For $x \in X'$ and $\gamma \in \T$, 
we have $X_{P_{(x,\gamma)}}=\overline{\Orb(x)}\subset X'$. 
Hence $I_{X'}\subset I_{X_{P_{(x,\gamma)}}} \subset P_{(x,\gamma)}$ 
by Lemma~\ref{Lem:Ieasy} and Lemma~\ref{Lem:IXICI}. 
Therefore we get $I_{X'}\subset I$. 
For $x \in X'$ and $\gamma \in \T$, 
$I \subset P_{(x,\gamma)}$ implies 
$X_I \supset X_{P_{(x,\gamma)}}=\overline{\Orb(x)}$. 
Hence we get $X_I \supset X'$. 
By Proposition~\ref{Prop:IXI} and Lemma~\ref{Lem:Ieasy}, 
we get $I=I_{X_I} \subset I_{X'}$. 
Thus we have $I_{X'} = I$. 
\end{proof}

\begin{corollary}\label{Cor:inter}
Every primitive ideals of $C^*(\Sigma)$ is 
an intersection of primitive ideals in the form $P_{(x,\gamma)}$. 
\end{corollary}

\section{$Y_I$, $I_Y$ and $I_{Y_I}=I$}\label{Sec:IYI} 

\begin{definition}
For an ideal $I$ of $C^*(\Sigma)$, 
we define $Y_I\subset X\times\T$ by 
\[
Y_I=\big\{(x,\gamma)\in X\times\T \mid 
I\subset P_{(x,\gamma)}\big\}.
\]
\end{definition}

\begin{lemma}\label{cap}
For two ideals $I_1,I_2$ of $C^*(\Sigma)$, 
we have $Y_{I_1\cap I_2}=Y_{I_1}\cup Y_{I_2}$.  
If $I_1\subset I_2$, we have $Y_{I_1}\supset Y_{I_2}$. 
\end{lemma}

\begin{proof}
To show $Y_{I_1\cap I_2}=Y_{I_1}\cup Y_{I_2}$, 
it suffices to show for $(x,\gamma)\in X\times\T$, 
$I_1\cap I_2 \subset P_{(x,\gamma)}$ 
if and only if $I_1\subset P_{(x,\gamma)}$ or $I_2\subset P_{(x,\gamma)}$. 
This follows from the fact that $P_{(x,\gamma)}$ is prime. 
The latter assertion is clear by definition. 
\end{proof}

\begin{definition}
For a subset $Y$ of $X\times\T$, 
we define an ideal $I_Y$ of $C^*(\Sigma)$ by 
\[
I_Y=\bigcap_{(x,\gamma)\in Y}P_{(x,\gamma)}.
\]
\end{definition}

\begin{lemma}
For two subsets $Y_1,Y_2$ of $X\times\T$, 
we have $I_{Y_1\cup Y_2}=I_{Y_1}\cap I_{Y_2}$.  
If $Y_1\subset Y_2$, we have $I_{Y_1}\supset I_{Y_2}$. 
\end{lemma}

\begin{proof}
Clear by definition. 
\end{proof}

\begin{proposition}\label{Prop:IYI}
For an ideal $I$ of $C^*(\Sigma)$, 
we have $I_{Y_I}=I$. 
\end{proposition}

\begin{proof}
It is known that $I$ is the intersection of 
all primitive ideals of $C^*(\Sigma)$ containing $I$ 
(\cite[II.6.5.3]{B}). 
By Corollary~\ref{Cor:inter}, 
all primitive ideals of $C^*(\Sigma)$ 
are the intersection of ideals 
in the form $P_{(x,\gamma)}$. 
We are done. 
\end{proof}

Thus $I\mapsto Y_I$ is an injective map 
from the set of ideals of $C^*(\Sigma)$ 
to the set of subsets of $X\times\T$. 
We will determine the image of this map 
in order to get the complete description of ideals of $C^*(\Sigma)$. 

\begin{definition}
Let $Y$ be a subset of $X\times\T$. 
For each $x\in X$, 
we set 
\[
Y_x=\{\gamma\in\T \mid (x,\gamma)\in Y\}\subset\T. 
\]
\end{definition}

\begin{definition}\label{Def:adm}
We say a subset $Y$ of $X\times\T$ is said to be 
{\em admissible} if 
\begin{enumerate}
\rom
\item $Y$ is a closed subset of $X\times\T$ 
with respect to the product topology, 
\item $Y_x=Y_{\sigma(x)}$ for all $x\in U$, 
\item $Y_{x_0}\neq \emptyset,\T$ implies that 
$x_0$ is periodic, $\zeta_{p(x_0)}Y_{x_0}=Y_{x_0}$, and 
there exists a neighborhood $V$ of $x_0$ such that 
all $x\in V$ with $l(x)\neq l(x_0)$ satisfies $Y_x=\emptyset$. 
\end{enumerate}
\end{definition}

\section{$Y_I$ is admissible}\label{Sec:YIad} 

We will show that $Y_I$ is admissible 
for all ideal $I$ of $C^*(\Sigma)$ 
(Proposition~\ref{Prop:YIinv}). 

For a periodic point $x_0\in X$ with $p:=p(x_0)$ 
and a positive integer $n$, 
we define a map $\sigma_n\colon 
\Orb(x_0)\times\Z/n\Z \to \Orb(x_0)\times\Z/n\Z$ 
as follows. 

We set $y_0=\sigma^{l(x_0)+p-1}(x_0)$. 
We define $d\colon \Orb(x_0)\to \{0,1\}$ 
by $d(x)=0$ for $x\in \Orb(x_0)\setminus \{y_0\}$ 
and $d(y_0)=1$, 
and a map $\sigma_n\colon 
\Orb(x_0)\times\Z/n\Z \to \Orb(x_0)\times\Z/n\Z$ 
by $\sigma_n(x,j)=(\sigma(x),j+d(x))$. 
Let $H_{x_0}^{(n)}$ be the Hilbert space whose complete orthonormal 
system is given by $\{\delta_{(x,j)}\}_{(x,j)\in \Orb(x_0)\times\Z/n\Z}$. 
In a similar way to define representations $\pi_{(x_0,\gamma_0)}$, 
we can define a representation 
$\pi_{(x_0,\gamma_0)}^{(n)}\colon C^*(\Sigma)\to B(H_{x_0}^{(n)})$ 
such that 
\begin{align*}
\pi_{(x_0,\gamma_0)}^{(n)}(t^0(f))\delta_{(x,j)}
&=f(x)\delta_{(x,j)}\\
\pi_{(x_0,\gamma_0)}^{(n)}(t^1(\xi))\delta_{(x,j)}
&=\gamma_0\sum_{\sigma_n((y,j'))=(x,j)}\xi(y)\delta_{(y,j')}
=\gamma_0\sum_{\sigma(y)=x}\xi(y)\delta_{(y,j-d(y))}. 
\end{align*}
By construction, we have $\pi_{(x_0,\gamma_0)}^{(1)}=\pi_{(x_0,\gamma_0)}$. 
We will see that $\pi_{(x_0,\gamma_0)}^{(n)}$ 
is not irreducible for $n\geq 2$ in Lemma~\ref{kerpi1}. 
For $x\in \Orb(x_0)$, 
let $c(x)\in\N$ be the smallest natural number 
satisfying $\sigma^{c(x)}(x)=y_0$. 
We have $c(\sigma(x))=c(x)-1$ for $x\in \Orb(x_0)\setminus \{y_0\}$, 
and $c(\sigma(y_0))=p-1=c(y_0)+p-1$. 
Hence we get $d(y)p=c(\sigma(y))-c(y)+1$ for all $y\in \Orb(x_0)$. 

\begin{lemma}\label{iotak}
For $k=0,1,\ldots,n-1$, 
the map $\iota_k\colon H_{x_0}\to H_{x_0}^{(n)}$ defined by 
\[
\iota_k(\delta_x)
=\frac{1}{\sqrt{n}}\sum_{j=0}^{n-1}
\zeta_{np}^{(jp-c(x))k}\delta_{(x,j)}, 
\]
is an isometry satisfying 
$\iota_k\circ \pi_{(x_0,\zeta_{np}^{k}\gamma_0)}(a)
=\pi_{(x_0,\gamma_0)}^{(n)}(a)\circ \iota_k$ 
for $a\in C^*(\Sigma)$. 
\end{lemma}

\begin{proof}
Clearly $\iota_k$ is isometric. 
For every $f\in C_0(X)$ and $x\in\Orb(x_0)$, 
we have 
\begin{align*}
\iota_k\big(\pi_{(x_0,\zeta_{np}^{k}\gamma_0)}(t^0(f))\delta_x\big)
&=\iota_k (f(x)\delta_x)
=f(x)\frac{1}{\sqrt{n}}
\sum_{j=0}^{n-1}\zeta_{np}^{(jp-c(x))k}\delta_{(x,j)}, 
\end{align*}
and 
\begin{align*}
\pi_{(x_0,\gamma_0)}^{(n)}(t^0(f))\big(\iota_k(\delta_x)\big)
&=\pi_{(x_0,\gamma_0)}^{(n)}(t^0(f))\bigg(
\frac{1}{\sqrt{n}}\sum_{j=0}^{n-1}
\zeta_{np}^{(jp-c(x))k}\delta_{(x,j)}\bigg)\\
&=\frac{1}{\sqrt{n}}\sum_{j=0}^{n-1}
\zeta_{np}^{(jp-c(x))k}f(x)\delta_{(x,j)}. 
\end{align*}
Hence 
$\iota_k\circ \pi_{(x_0,\zeta_{np}^{k}\gamma_0)}(a)
=\pi_{(x_0,\gamma_0)}^{(n)}(a)\circ \iota_k$ 
for $a\in t^0(C_0(X))$. 
For every $\xi\in C_c(U)$ and $x\in\Orb(x_0)$, 
we have 
\begin{align*}
\iota_k\big(\pi_{(x_0,\zeta_{np}^{k}\gamma_0)}(t^1(\xi))\delta_x\big)
&=\iota_k \bigg(\zeta_{np}^{k}\gamma_0\sum_{\sigma(y)=x}\xi(y)\delta_y\bigg)\\
&=\zeta_{np}^{k}\gamma_0\sum_{\sigma(y)=x}\xi(y)
\bigg(\frac{1}{\sqrt{n}}
\sum_{j=0}^{n-1}\zeta_{np}^{(jp-c(y))k}\delta_{(y,j)}\bigg)\\
&=\frac{1}{\sqrt{n}}\gamma_0\sum_{\sigma(y)=x}\xi(y)
\sum_{j=0}^{n-1}
\zeta_{np}^{(jp-c(y)+1)k}
\delta_{(y,j)},
\end{align*}
and 
\begin{align*}
\pi_{(x_0,\gamma_0)}^{(n)}(t^1(\xi))\big(\iota_k(\delta_x)\big)
&=\pi_{(x_0,\gamma_0)}^{(n)}(t^1(\xi))\bigg(\frac{1}{\sqrt{n}}
\sum_{j=0}^{n-1}\zeta_{np}^{(jp-c(x))k}\delta_{(x,j)}\bigg)\\
&=\frac{1}{\sqrt{n}}\sum_{j=0}^{n-1}\zeta_{np}^{(jp-c(x))k}
\bigg(\gamma_0\sum_{\sigma(y)=x}\xi(y)
\delta_{(y,j-d(y))}\bigg)\\ 
&=\frac{1}{\sqrt{n}}\gamma_0\sum_{\sigma(y)=x}\xi(y)
\sum_{j=0}^{n-1}\zeta_{np}^{(jp-c(x))k}
\delta_{(y,j-d(y))}\\
&=\frac{1}{\sqrt{n}}\gamma_0\sum_{\sigma(y)=x}\xi(y)
\sum_{j=0}^{n-1}
\zeta_{np}^{((j+d(y))p-c(x))k}
\delta_{(y,j)}\\
&=\frac{1}{\sqrt{n}}\gamma_0\sum_{\sigma(y)=x}\xi(y)
\sum_{j=0}^{n-1}
\zeta_{np}^{(jp-c(y)+1)k}
\delta_{(y,j)}
\end{align*}
where in the last equality 
we use the fact that $d(y)p=c(x)-c(y)+1$ 
for $y\in\Orb(x_0)$ with $\sigma(y)=x$. 
Hence 
$\iota_k\circ \pi_{(x_0,\zeta_{np}^{k}\gamma_0)}(a)
=\pi_{(x_0,\gamma_0)}^{(n)}(a)\circ \iota_k$ 
for $a\in t^1(C_c(U))$. 
Since $C^*(\Sigma)$ is generated by $t^0(C_0(X))\cup t^1(C_c(U))$, 
we have 
$\iota_k\circ \pi_{(x_0,\zeta_{np}^{k}\gamma_0)}(a)
=\pi_{(x_0,\gamma_0)}^{(n)}(a)\circ \iota_k$ 
for every $a\in C^*(\Sigma)$. 
\end{proof}

\begin{lemma}\label{kerpi1}
The representation $\pi_{(x_0,\gamma_0)}^{(n)}$ 
is unitarily equivalent to 
$\bigoplus_{k=0}^{n-1}\pi_{(x_0,\zeta_{np}^{k}\gamma_0)}$. 
Hence we have 
$\ker \pi_{(x_0,\gamma_0)}^{(n)}
=\bigcap_{k=0}^{n-1}P_{(x_0,\zeta_{np}^{k}\gamma_0)}$. 
\end{lemma}

\begin{proof}
For $k=0,1,\ldots,n-1$, 
let $\iota_k$ be the isometies in Lemma \ref{iotak}. 
Then for each $x\in \Orb(x_0)$, 
$\{\iota_k(\delta_x)\}_{k=0}^{n-1}$ 
are complete orthonormal system of 
the linear space spanned by $\{\delta_{(x,j)}\mid j\in \Z/n\Z\}$. 
Hence $\bigoplus_{k=0}^{n-1}\iota_k$ 
is a unitary from $\bigoplus_{k=0}^{n-1}H_{x_0}$ to $H_{x_0}^{(n)}$ 
which intertwines 
$\bigoplus_{k=0}^{n-1}\pi_{(x_0,\zeta_{np}^{k}\gamma_0)}$ 
and $\pi_{(x_0,\gamma_0)}^{(n)}$. 
This completes the proof. 
\end{proof}

Note that Lemma \ref{per1} implies 
$P_{(x_0,\zeta_{np}^{k}\gamma_0)}=P_{(x_0,\zeta_{np}^{k+n}\gamma_0)}$. 
Hence we get $\ker \pi_{(x_0,\gamma_0)}^{(n)}
=\bigcap_{k\in\Z}P_{(x_0,\zeta_{np}^{k}\gamma_0)}$.

\begin{lemma}\label{weakconv1}
Let $l,m\in\N$ with $m \geq 1$. 
Let $\{(x_\lambda,\gamma_\lambda)\}_{\lambda\in\Lambda}$ 
be a net in $X\times\T$ converges to $(x_0,\gamma_0)$ 
such that $l(x_\lambda)=l$ and $p(x_\lambda)=m$ 
for all $\lambda\in \Lambda$. 
Then we have the following. 
\begin{enumerate}
\renewcommand{\labelenumi}{{\rm (\roman{enumi})}}
\item $l(x_0)=l$ and $p(x_0)=m/n$ for some $n\in \N$. 
\item for each $(x,j)\in \Orb(x_0)\times \Z/n\Z$, 
there exist $\lambda_{(x,j)}\in \Lambda$ 
and $\varphi_\lambda(x,j)\in \Orb(x_\lambda)$ 
for $\lambda\succeq\lambda_{(x,j)}$ satisfying 
\begin{itemize}
\item $\lim_{\lambda}\varphi_\lambda(x,j)=x$ for all $(x,j)$, 
\item $\sigma(\varphi_\lambda(x,j))=\varphi_\lambda(\sigma_n(x,j))$ 
for all $(x,j)$ and all $\lambda$ with $\lambda\succeq\lambda_{(x,j)}$ 
and $\lambda\succeq\lambda_{\sigma_n(x,j)}$, 
\item for two distinct elements 
$(x,j),(x',j')\in \Orb(x_0)\times \Z/n\Z$, 
we have $\varphi_\lambda(x,j)\neq \varphi_\lambda(x',j')$ 
for sufficiently large $\lambda$. 
\end{itemize}
\item 
$\bigcap_{\lambda}P_{(x_\lambda,\gamma_\lambda)}\subset 
\ker\pi_{(x_0,\gamma_0)}^{(n)}$. 
\end{enumerate}
\end{lemma}

\begin{proof}
(i) 
For each $\lambda$, 
we have $\sigma^{l+m}(x_\lambda)=\sigma^{l}(x_\lambda)$. 
Hence we get $\sigma^{l+m}(x_0)=\sigma^{l}(x_0)$. 
This implies that $x_0$ is a periodic point 
whose period $p(x_0)$ divides $m$. 
Thus we can find $n\in\N$ with $p(x_0)=m/n$. 
We clearly have $l(x_0)\leq l$. 
To drive a contradiction, assume $l(x_0)=k<l$. 
Then we have $\sigma^{k+m}(x_{0})=\sigma^{k}(x_{0})$. 
Since $\sigma^{l-k}$ is locally homeomorphic, 
there exists a neighborhood $V$ 
of $\sigma^{k+m}(x_{0})=\sigma^{k}(x_{0})$ 
such that $\sigma^{l-k}$ is injective on $V$. 
We can find $\lambda$ 
such that $\sigma^{k+m}(x_{\lambda}),\sigma^{k}(x_{\lambda})\in V$. 
Since $k<l$, we have $\sigma^{k+m}(x_{\lambda})\neq \sigma^{k}(x_{\lambda})$. 
However, we have 
\[
\sigma^{l-k}\big(\sigma^{k+m}(x_{\lambda})\big)
=\sigma^{l+m}(x_{\lambda})=\sigma^{l}(x_\lambda)
=\sigma^{l-k}\big(\sigma^{k}(x_{\lambda})\big). 
\]
This is a contradiction. 
Hence $l(x_0)=l$. 

(ii) 
Let us set $p=p(x_0)$, 
and $x^{(i)}=\sigma^{l+i}(x_0)$ for $i=0,\ldots,p-1$. 
Take $(x,j)\in \Orb(x_0)\times \Z/n\Z$. 
Let $k:=l(x)$ which is the smallest natural number 
with $\sigma^{k+p}(x)=\sigma^{k}(x)$. 
Choose $i\in\{0,1,\ldots,p-1\}$ such that 
$\sigma^{k}(x)=x^{(i)}$. 
Choose a neighborhood $V$ of $x$ such that 
$\sigma^{k}$ is injective on $V$. 
Since $\sigma^{l+i+jp}(x_0)=x^{(i)}$, 
there exists $\lambda_{(x,j)}\in\Lambda$ such that 
we have $\sigma^{l+i+jp}(x_\lambda)\in \sigma^{k}(V)$ 
for all $\lambda\succeq \lambda_{(x,j)}$. 
For $\lambda\succeq \lambda_{(x,j)}$, 
define $\varphi_\lambda(x,j)$ to be 
the element in $V$ with 
$\sigma^{k}(\varphi_\lambda(x,j))=\sigma^{l+i+jp}(x_\lambda)$. 
Since the restriction of $\sigma^{k}$ to $V$ 
is a homeomorphism and 
$\lim_{\lambda}\sigma^{l+i+jp}(x_\lambda)=\sigma^{l+i+jp}(x_0)=x^{(i)}$, 
we get $\lim_{\lambda}\varphi_\lambda(x,j)=x$. 
By construction, 
it is easy to see that 
$\sigma(\varphi_\lambda(x,j))=\varphi_\lambda(\sigma_n(x,j))$ 
for $(x,j)\in (\Orb(x_0)\setminus\{x^{(p-1)}\})\times \Z/n\Z$ 
and $\lambda\in\Lambda$ with $\lambda\succeq \lambda_{(x,j)}$
and $\lambda\succeq \lambda_{\sigma_n(x,j)}$. 
For $(x^{(p-1)},j)\in \Orb(x_0)\times \Z/n\Z$, 
we have $\sigma_n(x^{(p-1)},j)=(x^{(0)},j+1)$. 
We get $\varphi_\lambda(x^{(p-1)},j)=\sigma^{l+p-1+jp}(x_\lambda)$ 
and $\varphi_\lambda(x^{(0)},j+1)=\sigma^{l+0+(j+1)p}(x_\lambda)$. 
Hence the equation 
$\sigma(\varphi_\lambda(x^{(p-1)},j))
=\varphi_\lambda(\sigma_n(x^{(p-1)},j))$ 
is easy to see when $j\neq n-1$, 
and follows from the fact that 
$\sigma^{l+np}(x_\lambda)=\sigma^{l}(x_\lambda)$ 
when $j=n-1$. 
Take two distinct elements 
$(x,j),(x',j')\in \Orb(x_0)\times \Z/n\Z$. 
If $x\neq x'$, then 
$\varphi_\lambda(x,j)\neq \varphi_\lambda(x',j')$ 
for sufficiently large $\lambda$ 
because $\lim_{\lambda}\varphi_\lambda(x,j)=x$ 
and $\lim_{\lambda}\varphi_\lambda(x',j')=x'$. 
If $x=x'$, then $j\neq j'$. 
Take $k:=l(x)$ and $i\in\{0,1,\ldots,p-1\}$ with 
$\sigma^{k}(x)=x^{(i)}$. 
For every $\lambda$ 
we have $\sigma^{l+i+jp}(x_\lambda)\neq \sigma^{l+i+j'p}(x_\lambda)$ 
because $0<|jp-j'p|<m$.
Hence $\varphi_\lambda(x,j)\neq \varphi_\lambda(x',j')$ 
when they are defined 
because $\sigma^{k}(\varphi_\lambda(x,j))=\sigma^{l+i+jp}(x_\lambda)$ 
and $\sigma^{k}(\varphi_\lambda(x',j'))=\sigma^{l+i+j'p}(x_\lambda)$. 

(iii) 
Take $(x,j),(x',j')\in \Orb(x_0)\times \Z/n\Z$. 
We will show that for all $a\in C^*(\Sigma)$ we have 
\[
\lim_{\lambda}
\ip{\delta_{\varphi_\lambda(x',j')}}{\pi_{(x_\lambda,\gamma_\lambda)}(a)
\delta_{\varphi_\lambda(x,j)}}_{x_\lambda}
=
\ip{\delta_{(x',j')}}{\pi_{(x_0,\gamma_0)}^{(n)}(a)\delta_{(x,j)}}_{x_0}^{(n)},
\]
where $\ip{\cdot}{\cdot}_{x_0}^{(n)}$ is the inner product of $H_{x_0}^{(n)}$ 
which is linear in the second variable. 
To do so, 
it suffices to assume that 
$a=\sum_{k=1}^K t^{n_k}(\xi_k)t^{m_k}(\eta_k)^*\in C^*(\Sigma)$ 
for some $K, n_k,m_k\in\N$, $\xi_k\in C_c(U_{n_k})$ 
and $\eta_k\in C_c(U_{m_k})$ for $k=1,\ldots,K$ by Lemma~\ref{Lem:cspa}. 

We can find $\lambda_0\in\Lambda$ such that 
\begin{itemize}
\item $\lambda_0\succeq\lambda_{(x,j)}$ 
and $\lambda_0\succeq\lambda_{(x',j')}$, 
\item $\lambda_0\succeq\lambda_{\sigma_n^{m_k}(x,j)}$ 
and $\lambda_0\succeq\lambda_{\sigma_n^{n_k}(x',j')}$ 
for $k=1,\ldots,K$, 
\item $\varphi_\lambda$ is injective on 
$\{\sigma_n^{m_k}(x,j),\sigma_n^{n_k}(x',j')\mid k=1,\ldots,K\}$ 
for all $\lambda\succeq \lambda_0$. 
\end{itemize}
For $\lambda\succeq \lambda_0$, 
we have 
\begin{align*}
\pi_{(x_\lambda,\gamma_\lambda)}(a)\delta_{\varphi_{\lambda(x,j)}}
&=\sum_{k=1}^K \pi_{(x_\lambda,\gamma_\lambda)}
\big(t^{n_k}(\xi_k)t^{m_k}(\eta_k)^*\big)\delta_{\varphi_\lambda(x,j)}\\
&=\sum_{k=1}^K \pi_{(x_\lambda,\gamma_\lambda)}(t^{n_k}(\xi_k))
\big(\gamma_\lambda^{-m_k}\overline{\eta_k}(\varphi_\lambda(x,j))
\delta_{\sigma^{m_k}(\varphi_\lambda(x,j))}\big)\\
&=\sum_{k=1}^K
\sum_{\sigma^{n_k}(y)=\sigma^{m_k}(\varphi_\lambda(x,j))}\!\!
\gamma_\lambda^{n_k}\xi_k(y)
\gamma_\lambda^{-m_k}\overline{\eta_k}(\varphi_\lambda(x,j))
\delta_{y}. 
\end{align*}
Hence we get 
\[
\ip{\delta_{\varphi_\lambda(x',j')}}{\pi_{(x_\lambda,\gamma_\lambda)}(a)
\delta_{\varphi_\lambda(x,j)}}_{x_\lambda}
=\sum_{k\in N_\lambda}
\gamma_\lambda^{n_k-m_k}\xi_k(\varphi_\lambda(x',j'))
\overline{\eta_k}(\varphi_\lambda(x,j)),
\]
where 
\[
N_\lambda=\big\{k\in\{1,\ldots,K\}\ \big|\ 
\sigma^{n_k}(\varphi_\lambda(x',j'))=\sigma^{m_k}(\varphi_\lambda(x,j))\big\}. 
\]
Similarly, we get 
\[
\ip{\delta_{(x',j')}}{\pi_{(x_0,\gamma_0)}^{(n)}(a)\delta_{(x,j)}}_{x_0}^{(n)}
=\sum_{k\in N_0}
\gamma_0^{n_k-m_k}\xi_k(x')\overline{\eta_k}(x),
\]
where 
\[
N_0=\big\{k\in\{1,\ldots,K\}\ \big|\ 
\sigma_n^{n_k}(x',j')=\sigma_n^{m_k}(x,j)\big\}. 
\]
Since 
\[
\sigma^{n_k}(\varphi_\lambda(x',j'))
=\varphi_\lambda(\sigma_n^{n_k}(x',j')),\quad 
\sigma^{m_k}(\varphi_\lambda(x,j))
=\varphi_\lambda(\sigma_n^{m_k}(x,j))
\]
and $\varphi_\lambda$ is injective on 
$\{\sigma_n^{m_k}(x,j),\sigma_n^{n_k}(x',j')\mid k=1,\ldots,K\}$, 
we have $N_\lambda=N_0$ for $\lambda\succeq \lambda_0$. 
Since $\lim_{\lambda}\varphi_\lambda(x,j)=x$ 
and $\lim_{\lambda}\varphi_\lambda(x',j')=x'$, 
we have 
\[
\lim_{\lambda}
\ip{\delta_{\varphi_\lambda(x',j')}}{\pi_{(x_\lambda,\gamma_\lambda)}(a)
\delta_{\varphi_\lambda(x,j)}}_{x_\lambda}
=\ip{\delta_{(x',j')}}{\pi_{(x_0,\gamma_0)}^{(n)}(a)\delta_{(x,j)}}_{x_0}^{(n)}.
\]
Thus we get this equality 
for all $(x,j),(x',j')\in\Orb(x_0)\times \Z/n\Z$ 
and all $a\in C^*(\Sigma)$. 
Hence for $a\in \bigcap_{\lambda}P_{(x_\lambda,\gamma_\lambda)}$ 
we have $\pi_{(x_0,\gamma_0)}^{(n)}(a)=0$. 
We are done. 
\end{proof}

\begin{corollary}\label{closed}
Under the condition in Lemma \ref{weakconv1}, 
we have 
\[
\bigcap_{\lambda}P_{(x_\lambda,\gamma_\lambda)}
\subset P_{(x_0,\gamma_0)}. 
\] 
\end{corollary}

\begin{proof}
This follows from Lemma \ref{weakconv1} (iii) and Lemma \ref{kerpi1}. 
\end{proof}

For a periodic point $x_0\in X$ with $p := p(x_0)$, 
we define a representation 
$\pi_{(x_0,\gamma_0)}^{(\infty)}\colon C^*(\Sigma)\to B(H_{x_0}^{(\infty)})$ 
in a similar way to the definition of 
$\pi_{(x_0,\gamma_0)}^{(n)}$ 
where $H_{x_0}^{(\infty)}$ is the Hilbert space 
whose complete orthonormal system is given by 
$\{\delta_{(x,j)}\}_{(x,j)\in \Orb(x_0)\times\Z}$ 
using the map $\sigma_\infty\colon \Orb(x_0)\times\Z\to \Orb(x_0)\times\Z$ 
defined by $\sigma_\infty(x,j)=(\sigma(x),j+d(x))$. 
Let us define a probability space 
\[
[1,\zeta_p)=\{e^{2\pi i\theta/p}\mid 0\leq \theta <1\}
\]
on which we consider the measure $d\gamma$ defined from 
the Lebesgue measure via the bijection 
$[0,1)\ni\theta\mapsto e^{2\pi i\theta/p}\in [1,\zeta_p)$. 
The Hilbert space $L^2([1,\zeta_p),H_{x_0})$ 
of the all $H_{x_0}$-valued 
second power integrable functions on $[1,\zeta_p)$ 
has spanned by the elements in the form 
$\eta\delta_x$ with $\eta\in L^2([1,\zeta_p))$ 
and $x\in\Orb(x_0)$. 
We define a representation 
$\pi_{(x_0,[\gamma_0,\zeta_p\gamma_0))}\colon 
C^*(\Sigma)\to B(L^2([1,\zeta_p),H_{x_0}))$ by 
\[
\big(\pi_{(x_0,[\gamma_0,\zeta_p\gamma_0))}(a)
\varPsi \big)(\gamma)
=\pi_{(x_0,\gamma\gamma_0)}(a)\varPsi(\gamma)\in H_{x_0} 
\]
for $\varPsi\in L^2([1,\zeta_p),H_{x_0})$ 
and $\gamma\in [1,\zeta_p)$. 

\begin{lemma}\label{iota}
The map $\iota\colon L^2([1,\zeta_p),H_{x_0})\to H_{x_0}^{(\infty)}$ 
defined by 
\[
\iota(\eta\delta_x)
=\sum_{j\in\Z}\bigg(
\int_{[1,\zeta_p)}
\gamma^{jp-c(x)}\eta(\gamma)d\gamma \bigg)\delta_{(x,j)}, 
\]
for $\eta\in L^2([1,\zeta_p))$ and $x\in\Orb(x_0)$ 
is a unitary satisfying 
$\iota\circ \pi_{(x_0,[\gamma_0,\zeta_p\gamma_0))}(a)
=\pi_{(x_0,\gamma_0)}^{(\infty)}(a)\circ \iota$ 
for $a\in C^*(\Sigma)$. 
\end{lemma}

\begin{proof}
For each $x\in\Orb(x_0)$, 
the map 
$L^2([1,\zeta_p))\ni\eta(\cdot)
\mapsto (\cdot)^{-c(x)}\eta(\cdot)\in L^2([1,\zeta_p))$ 
is a unitary. 
Since $[1,\zeta_p)\ni \gamma \to \gamma^p\in\T$ is an isomorphism 
between probability spaces, 
the natural isomorphism $L^2(\T)\cong\ell^2(\Z)$ 
of the Fourier transform shows that 
the map 
\[
L^2([1,\zeta_p))\ni\eta\mapsto 
\sum_{j\in\Z}\bigg(
\int_{[1,\zeta_p)}
\gamma^{jp-c(x)}\eta(\gamma)d\gamma \bigg)\delta_{(x,j)}
\]
is an isomorphism onto the Hilbert subspace spanned 
by $\{\delta_{(x,j)}\}_{j\in\Z}$. 
Thus $\iota$ is a unitary. 

Take $\xi\in C_c(U)$, $\eta\in L^2([1,\zeta_p))$ and $x\in\Orb(x_0)$. 
For each $\gamma\in [1,\zeta_p)$, 
we obtain 
\begin{align*}
\Big(\pi_{(x_0,[\gamma_0,\zeta_p\gamma_0))}(t^1(\xi))
\eta\delta_x\Big)(\gamma)
&=\pi_{(x_0,\gamma\gamma_0)}(t^1(\xi))(\eta(\gamma)\delta_x)\\
&=\gamma\gamma_0\sum_{\sigma(y)=x}\xi(y)\eta(\gamma)\delta_y. 
\end{align*}
If we denote by $z\in L^2([1,\zeta_p))$ 
the identity function $\gamma\mapsto\gamma$, 
then we have 
\[
\pi_{(x_0,[\gamma_0,\zeta_p\gamma_0))}(t^1(\xi))
\eta\delta_x
=\gamma_0\sum_{\sigma(y)=x}\xi(y) \eta z \delta_y. 
\]
Hence 
\begin{align*}
\iota\big(\pi_{(x_0,[\gamma_0,\zeta_p\gamma_0))}(t^1(\xi))\eta\delta_x\big)
&=\iota \bigg(\gamma_0\sum_{\sigma(y)=x}\xi(y) \eta z \delta_y\bigg)\\
&=\gamma_0\sum_{\sigma(y)=x}\xi(y)
\sum_{j\in\Z}\bigg(
\int_{[1,\zeta_p)}
\gamma^{jp-c(y)}\eta(\gamma)\gamma d\gamma \bigg)\delta_{(y,j)}\\
&=\gamma_0\sum_{\sigma(y)=x}\sum_{j\in\Z}
\bigg(
\int_{[1,\zeta_p)}
\gamma^{jp-c(y)+1}\eta(\gamma)d\gamma \bigg)\xi(y)\delta_{(y,j)}. 
\end{align*}
On the other hand, we get 
\begin{align*}
\pi_{(x_0,\gamma_0)}^{(\infty)}(t^1(\xi))\big(\iota(\eta\delta_x)\big)
&=\pi_{(x_0,\gamma_0)}^{(\infty)}(t^1(\xi))
\bigg(\sum_{j\in\Z}
\bigg(\int_{[1,\zeta_p)}
\gamma^{jp-c(x)}\eta(\gamma)d\gamma \bigg)\delta_{(x,j)}\bigg)\\
&=\sum_{j\in\Z}
\bigg(\int_{[1,\zeta_p)}
\gamma^{jp-c(x)}\eta(\gamma)d\gamma \bigg)
\gamma_0\sum_{\sigma(y)=x}\xi(y)\delta_{(y,j-d(y))}\\
&=\gamma_0\sum_{\sigma(y)=x}\sum_{j\in\Z}
\bigg(\int_{[1,\zeta_p)}
\gamma^{jp-c(x)}\eta(\gamma)d\gamma \bigg)
\xi(y)\delta_{(y,j-d(y))}\\
&=\gamma_0\sum_{\sigma(y)=x}\sum_{j\in\Z}
\bigg(\int_{[1,\zeta_p)}
\gamma^{(j+d(y))p-c(x)}\eta(\gamma)d\gamma \bigg)
\xi(y)\delta_{(y,j)}\\
&=\gamma_0\sum_{\sigma(y)=x}\sum_{j\in\Z}
\bigg(\int_{[1,\zeta_p)}
\gamma^{jp-c(y)+1}\eta(\gamma)d\gamma \bigg)
\xi(y)\delta_{(y,j)}. 
\end{align*}
Hence we have 
$\iota\circ \pi_{(x_0,[\gamma_0,\zeta_p\gamma_0))}(a)
=\pi_{(x_0,\gamma_0)}^{(\infty)}(a)\circ \iota$ 
for all $a\in t^1(C_c(U))$. 
The proof of this equality for $a\in t^0(C_0(X))$ 
is much easier. 
Therefore we get the equality for all $a\in C^*(\Sigma)$. 
\end{proof}

\begin{corollary}\label{kerinf}
We have 
$\ker \pi_{(x_0,\gamma_0)}^{(\infty)}
=\bigcap_{\gamma\in\T}P_{(x_0,\gamma)}$. 
\end{corollary}

\begin{proof}
By Lemma \ref{iota}, 
we have 
\[
\ker \pi_{(x_0,\gamma_0)}^{(\infty)}
=\ker \pi_{(x_0,[\gamma_0,\zeta_p\gamma_0))}. 
\]
Since $\gamma\mapsto \pi_{(x_0,\gamma)}$ is 
pointwise norm continuous, 
we have 
\[
\ker \pi_{(x_0,[\gamma_0,\zeta_p\gamma_0))}
=\bigcap_{\gamma\in [\gamma_0,\zeta_p\gamma_0)}P_{(x_0,\gamma)}. 
\]
By Lemma \ref{per1}, 
we get 
\[
\bigcap_{\gamma\in [\gamma_0,\zeta_p\gamma_0)}P_{(x_0,\gamma)} 
=\bigcap_{\gamma\in\T}P_{(x_0,\gamma)}. 
\]
This completes the proof. 
\end{proof}

As usual, 
we consider $\N\cup\{\infty\}$ 
as the one-point compactification of 
the discrete set $\N$. 

\begin{lemma}\label{weakconv2}
Let $\{(x_\lambda,\gamma_\lambda)\}_{\lambda\in\Lambda}$ 
be a net in $X\times\T$ converges to $(x_0,\gamma_0)$ 
such that $l(x_\lambda)+p(x_\lambda)$ converges to $\infty$. 
Then we have 
$\bigcap_{\lambda}P_{(x_\lambda,\gamma_\lambda)}\subset P_{(x_0,\gamma)}$ 
for all $\gamma\in\T$. 
\end{lemma}

\begin{proof}
We divide the proof into two cases; 
the case that $x_0$ is periodic 
and the case that $x_0$ is aperiodic. 
In the both cases, 
the proofs go in a similar way to the proof of Lemma \ref{weakconv1}, 
and so we just sketch the proofs. 

First consider the case that $x_0$ is periodic 
with period $p(x_0)=p$. 
We first show that 
for each $(x,j)\in \Orb(x_0)\times \Z$, 
there exist $\lambda_{(x,j)}\in \Lambda$ 
and $\varphi_\lambda(x,j)\in \Orb(x_\lambda)$ 
for $\lambda\succeq\lambda_{(x,j)}$ satisfying 
\begin{itemize}
\item $\lim_{\lambda}\varphi_\lambda(x,j)=x$, 
\item $\sigma(\varphi_\lambda(x,j))=\varphi_\lambda(\sigma_\infty(x,j))$. 
\item for two distinct elements 
$(x,j),(x',j')\in \Orb(x_0)\times \Z$, 
we have $\varphi_\lambda(x,j)\neq \varphi_\lambda(x',j')$ 
for sufficiently large $\lambda$. 
\end{itemize}
We set $x'_0 := \sigma^{l(x_0)}(x_0)$. 
For $(x'_0,0)\in \Orb(x_0)\times \Z$, 
we set $\varphi_\lambda(x'_0,0)=\sigma^{l(x_0)}(x_\lambda)$ for all $\lambda$. 
There exists $\lambda_{(x'_0,1)} \in \Lambda$ 
such that $\varphi_\lambda(x'_0,0) \in U$ 
for all $\lambda\succeq \lambda_{(x'_0,0)}$. 
Set $\varphi_\lambda(\sigma(x'_0),0)=\sigma(\varphi_\lambda(x'_0,0))$ 
for all $\lambda\succeq \lambda_{(x'_0,0)}$. 
Similarly for $k=2,3,\ldots,p-1$ 
$\varphi_\lambda(\sigma^k(x'_0),0)=
\sigma(\varphi_\lambda(\sigma^{k-1}(x'_0),0))$ 
for sufficiently large $\lambda$. 
We also set for sufficiently large $\lambda$, 
$\varphi_\lambda(x'_0,1)=
\sigma(\varphi_\lambda(\sigma^{p-1}(x'_0),0))$ 
and $\varphi_\lambda(\sigma^k(x'_0),1)=
\sigma(\varphi_\lambda(\sigma^{k-1}(x'_0),1))$ 
for $k=1,2,\ldots,p-1$. 
For $j=2,3,\ldots$, 
we set for sufficiently large $\lambda$,  
$\varphi_\lambda(x'_0,j)=
\sigma(\varphi_\lambda(\sigma^{p-1}(x'_0),j-1))$ 
and $\varphi_\lambda(\sigma^k(x'_0),j)=
\sigma(\varphi_\lambda(\sigma^{k-1}(x'_0),j))$ 
for $k=1,2,\ldots,p-1$. 
Next we choose a neighborhood $V$ of $x'_0$ 
such that $\sigma^{p}$ is injective on $V$. 
We set $\lambda_{(x'_0,-1)} \in \Lambda$ 
such that $\varphi_\lambda(x'_0,0) \in \sigma^p(V)$ 
for all $\lambda\succeq \lambda_{(x'_0,-1)}$. 
We set $\varphi_\lambda(x'_0,-1) \in V$ 
such that $\sigma^p(\varphi_\lambda(x'_0,-1))=\varphi_\lambda(x'_0,0)$ 
and set $\varphi_\lambda(\sigma^k(x'_0),-1)=
\sigma(\varphi_\lambda(\sigma^{k-1}(x'_0),-1))$ 
for $k=1,2,\ldots,p-1$. 
We set $\lambda_{(x'_0,-2)} \in \Lambda$ 
such that $\varphi_\lambda(x'_0,-1) \in \sigma^p(V)$ 
for all $\lambda\succeq \lambda_{(x'_0,-2)}$. 
We set $\varphi_\lambda(x'_0,-2) \in V$ 
such that $\sigma^p(\varphi_\lambda(x'_0,-2))=\varphi_\lambda(x'_0,-1)$ 
and set $\varphi_\lambda(\sigma^k(x'_0),-2)=
\sigma(\varphi_\lambda(\sigma^{k-1}(x'_0),-2))$ 
for $k=1,2,\ldots,p-1$. 
For $j=3,4,\ldots$, 
we set $\lambda_{(x'_0,-j)} \in \Lambda$ 
such that $\varphi_\lambda(x'_0,-j+1) \in \sigma^p(V)$ 
for all $\lambda\succeq \lambda_{(x'_0,-j)}$. 
We set $\varphi_\lambda(x'_0,-j) \in V$ 
such that $\sigma^p(\varphi_\lambda(x'_0,-j))=\varphi_\lambda(x'_0,-j+1)$ 
and set $\varphi_\lambda(\sigma^k(x'_0),-j)=
\sigma(\varphi_\lambda(\sigma^{k-1}(x'_0),-j))$ 
for $k=1,2,\ldots,p-1$. 
We have defined $\varphi_\lambda(x,j)$ for all 
$(x,j)\in \Orb(x_0)\times \Z$ with $l(x)=0$. 
For $(x,j)\in \Orb(x_0)\times \Z$ with $l(x)\geq 1$, 
we choose a neighborhood $V$ of $x$ 
such that $\sigma^{k}$ is injective on $V$ where $k:=l(x)$. 
We set $\lambda_{(x,j)} \in \Lambda$ 
such that $\varphi_\lambda(\sigma^k(x),j) \in \sigma^k(V)$ 
for all $\lambda\succeq \lambda_{(x,j)}$. 
We set $\varphi_\lambda(x,j) \in V$ 
such that $\sigma^k(\varphi_\lambda(x,j))=\varphi_\lambda(\sigma^k(x),j)$. 
It is very similar as in the proof of Lemma \ref{weakconv1} (ii) 
to check that $\varphi_\lambda(x,j)$'s satisfy the desired conditions 
except the proof of the statement 
that for $x\in\Orb(x_0)$ and $j,j'\in\Z$ with $j < j'$, 
we have $\varphi_\lambda(x,j)\neq \varphi_\lambda(x,j')$ 
for sufficiently large $\lambda$. 
This follows from the fact that, 
we have 
\begin{align*}
\sigma^{l(x)+N}(\varphi_\lambda(x,j))&=\sigma^{N'}(x_\lambda),&
\sigma^{l(x)+N}(\varphi_\lambda(x,j'))&=\sigma^{N'+p(j'-j)}(x_\lambda) 
\end{align*}
for sufficiently large $N$ and some $N'$. 
Since $l(x_\lambda)+p(x_\lambda)$ converges to $\infty$ 
we have $\sigma^{N'}(x_\lambda) \neq \sigma^{N'+p(j'-j)}(x_\lambda)$
for sufficiently large $\lambda$. 
This shows $\varphi_\lambda(x,j)\neq \varphi_\lambda(x,j')$ 
for sufficiently large $\lambda$. 

Now, in a similar to the 
proof of Lemma \ref{weakconv1} (iii), 
we can show that for all $(x,j),(x',j')\in \Orb(x_0)\times \Z$ 
and all $a\in C^*(\Sigma)$, 
we get 
\[
\lim_{\lambda}
\ip{\delta_{\varphi_\lambda(x',j')}}{\pi_{(x_\lambda,\gamma_\lambda)}(a)
\delta_{\varphi_\lambda(x,j)}}_{x_\lambda}
=\ip{\delta_{(x',j')}}{\pi_{(x_0,\gamma_0)}^{(\infty)}(a)\delta_{(x,j)}}_{x_0}^{(\infty)}
\]
where $\ip{\cdot}{\cdot}_{x_0}^{(\infty)}$ is 
the inner product of $H_{x_0}^{(\infty)}$. 
Hence we have 
$\bigcap_{\lambda}P_{(x_\lambda,\gamma_\lambda)}
\subset \ker\pi_{(x_0,\gamma_0)}^{(\infty)}$. 
We finish the proof for the case that $x_0$ is periodic 
by Corollary \ref{kerinf}. 

For the case that $x_0$ is aperiodic, 
we can similarly construct $\varphi_\lambda(x)\in \Orb(x_\lambda)$ 
for $x\in \Orb(x_0)$ and sufficiently large $\lambda\in\Lambda$, 
so that we get 
\[
\lim_{\lambda}
\ip{\delta_{\varphi_\lambda(x')}}{\pi_{(x_\lambda,\gamma_\lambda)}(a)
\delta_{\varphi_\lambda(x)}}_{x_\lambda}
=\ip{\delta_{x'}}{\pi_{(x_0,\gamma_0)}(a)\delta_{x}}_{x_0},
\]
for all $x,x'\in \Orb(x_0)$ 
and all $a\in C^*(\Sigma)$. 
Hence we have 
$\bigcap_{\lambda}P_{(x_\lambda,\gamma_\lambda)}\subset P_{(x_0,\gamma_0)}$. 
Since $P_{(x_0,\gamma)}=P_{(x_0,\gamma_0)}$ for all $\gamma\in\T$ 
by Lemma~\ref{aper}, 
we finish the case that $x_0$ is aperiodic. 
\end{proof}

\begin{proposition}\label{Prop:YIinv}
For an ideal $I$, 
the set $Y_I$ is admissible. 
\end{proposition}

\begin{proof}
Take a net $\{(x_\lambda,\gamma_\lambda)\}_{\lambda\in\Lambda}$ 
in $Y_I$ converges to $(x_0,\gamma_0)\in X\times\T$. 
By replacing the net to a subnet if necessary, 
we may assume that 
either $l(x_\lambda)$ and $p(x_\lambda)$ are finite and constant 
or $l(x_\lambda)+p(x_\lambda)$ converges $\infty$. 
When $l(x_\lambda)$ and $p(x_\lambda)$ are finite and constant, 
we have 
\[
I\subset \bigcap_{\lambda}P_{(x_\lambda,\gamma_\lambda)}
\subset P_{(x_0,\gamma_0)}
\] 
by Corollary \ref{closed}. 
When $l(x_\lambda)+p(x_\lambda)$ converges $\infty$, 
we have 
\[
I\subset \bigcap_{\lambda}P_{(x_\lambda,\gamma_\lambda)}
\subset P_{(x_0,\gamma_0)}
\] 
by Lemma \ref{weakconv2}. 
Hence we get $(x_0,\gamma_0)\in Y_I$. 
This shows that $Y_I$ is closed. 
By Lemma \ref{inv0}, 
$Y_I$ satisfies the condition (ii) 
in Definition \ref{Def:adm}. 

Take $x_0\in X$ with $(Y_I)_{x_0}\neq \emptyset,\T$. 
By Lemma \ref{aper}, 
we see that $x_0$ is periodic, 
and by Lemma \ref{per1} 
we see that $\zeta_{p(x_0)}(Y_I)_{x_0}=(Y_I)_{x_0}$. 
To the contrary, 
suppose that for any neighborhood $V$ of $x_0$ 
we can find $x\in V$ with $l(x)\neq l(x_0)$
satisfying $(Y_I)_x\neq\emptyset$. 
Then we can find a net $\{(x_\lambda,\gamma_\lambda)\}_{\lambda\in\Lambda}$ 
in $Y_I$ with $l(x_\lambda) \neq l(x_0)$ 
which converges $(x_0,\gamma_0)$ for some $\gamma_0\in\T$. 
By taking a subnet if necessary, 
we may assume that either 
$l(x_\lambda)$ and $p(x_\lambda)$ are finite and constant 
or $l(x_\lambda)+p(x_\lambda)$ converges to $\infty$ 
. 
If $l(x_\lambda)$ and $p(x_\lambda)$ are finite and constant, 
then $l(x_0)=l(x_\lambda)$ by Lemma~\ref{weakconv1} (i). 
This is a contradiction. 
If $l(x_\lambda)+p(x_\lambda)$ converges to $\infty$ 
then we have $(Y_I)_{x_0}=\T$ by Lemma~\ref{weakconv2}. 
This is also a contradiction. 
Therefore we can find a neighborhood $V$ of $x_0$ such that 
all $x\in V$ with $l(x)\neq l(x_0)$ satisfies $(Y_I)_x=\emptyset$. 
We have proved that $Y_I$ is admissible. 
\end{proof}

\section{A proof of $Y_{I_Y}=Y$}\label{Sec:YIY}

Take an admissible subset $Y$ of $X\times\T$. 
In this section we will prove that $Y_{I_Y}=Y$. 
By definition, we have $Y_{I_Y}\supset Y$. 
Thus all we have to do is to prove 
the other inclusion $Y_{I_Y}\subset Y$. 
We set $X' :=\{x\in X\mid Y_x\neq\emptyset\}$. 

\begin{lemma}
The set $X'$ is closed and $\sigma$-invariant. 
\end{lemma}

\begin{proof}
The set $X'$ is closed because $Y$ is closed 
and $\T$ is compact. 
By the condition (ii) in Definition~\ref{Def:adm}, 
$X'$ is $\sigma$-invariant. 
\end{proof}

Take $x_0\notin X'$ and $\gamma_0\in\T$. 
We can find $f\in C_0(X)$ 
such that $f(x_0)=1$ and $f(x)=0$ for all $x\in X'$. 
Then we have $t^0(f)\in I_Y$ 
because $\pi_{(x,\gamma)}(t^0(f))=0$ 
for all $x\in X'$ and all $\gamma\in\T$. 
However $t^0(f)\notin P_{(x_0,\gamma_0)}$ 
because $\pi_{(x_0,\gamma_0)}(t^0(f))\delta_{x_0}=\delta_{x_0}$. 
This implies $(x_0,\gamma_0)\notin Y_{I_Y}$. 
Therefore we have $Y_{I_Y}\subset X'\times \T$. 

Take $(x_0,\gamma_0)\in (X'\times \T)\setminus Y$, 
and we will show that $(x_0,\gamma_0)\notin Y_{I_Y}$ 
which is equivalent to $I_Y \not\subset P_{(x_0,\gamma_0)}$. 
To do so, it suffices to construct $a \in C^*(\Sigma)$ 
such that $\pi_{(x_0,\gamma_0)}(a)\neq 0$ and 
$\pi_{(x,\gamma)}(a) = 0$ for all $(x,\gamma) \in Y$. 
Note that the representation $\pi_{(x_0,\gamma_0)}$ factors through 
the natural surjection $C^*(\Sigma)\to C^*(\Sigma')$ 
where $\Sigma'=(X',\sigma|_{U\cap X'})$. 
The induced representation of $C^*(\Sigma')$ 
can be denoted by the same notation $\pi_{(x_0,\gamma_0)}$ 
considered as $x_0 \in X'$. 
Similarly for each $(x,\gamma) \in Y$, 
the representation $\pi_{(x,\gamma)}\colon 
C^*(\Sigma) \to B(H_{x_0})$ factors through 
the representation $\pi_{(x,\gamma)}\colon 
C^*(\Sigma') \to B(H_{x_0})$ 
considered as $x \in X'$. 
Hence to finish the proof it suffices to 
construct $a \in C^*(\Sigma')$ 
such that $\pi_{(x_0,\gamma_0)}(a)\neq 0$ and 
$\pi_{(x,\gamma)}(a) = 0$ for all $(x,\gamma) \in Y$. 

We have $Y_{x_0}\neq\emptyset,\T$. 
Hence by the condition (iii) 
in Definition~\ref{Def:adm}, 
there exists an open neighborhood $V$ of $x_0$ such that 
$l(x)=l(x_0)$ for all $x\in X'\cap V$. 
We set $x_0'=\sigma^{l(x_0)}(x_0)$ 
which satisfies $l(x_0')=0$. 
Set $V'=\sigma^{l(x_0)}(V)\cap X'$ 
which is an open neighborhood of $x_0'\in X'$ 
such that all $x'\in V'$ satisfies $l(x')=0$. 
We have $(x_0',\gamma_0)\notin Y$ because $(x_0,\gamma_0)\notin Y$. 
Since $Y$ is closed, 
we can find an open neighborhood $W'$ of $x_0'\in X'$  
and an open neighborhood $Z$ of $\gamma_0\in\T$ such that 
$(W'\times Z)\cap Y=\emptyset$. 
We may assume that $W' \subset V'$. 
For each $x\in W'$, 
we have $Y_x\cap Z=\emptyset$ and $\zeta_{p(x)}Y_x=Y_x$. 
Hence $\sup_{x\in W'}p(x)<\infty$. 
Thus we can find $N\in\N$ 
such that $\sigma^N(x)=x$ for all $x\in W'$ 
(recall that $l(x)=0$ for $x\in W'$). 
Let us set $W=\bigcup_{k=0}^{N-1}\sigma^k(W')$ 
which is an open subset of $X'$ satisfying $\sigma(W)=W$, 
and we denote by $\sigma_W\colon W\to W$ 
the restriction of $\sigma$ to $W$, 
which is a homeomorphism satisfying 
$\sigma_W^N=\id_{W}$. 
By the condition (ii) in Definition~\ref{Def:adm}, 
we have 
$(W\times Z)\cap Y=\emptyset$. 

For $f\in C_0(W)$ and a negative integer $n$, 
we define $t^{n}(f)\in C^*(\Sigma)$ 
by $t^{n}(f)=t^{-n}(\overline{f}\circ\sigma_W^{-n})^*$. 

\begin{lemma}\label{Lem:tnf1}
Let $f\in C_0(W)$. 
For $(x,\gamma)\in X'\times\T$, $y\in \Orb(x)$ and $n\in\Z$,
we have 
\begin{equation*}
\pi_{(x,\gamma)}\big(t^{n}(f)\big)\delta_y
=\begin{cases}
\gamma_0^n f(\sigma_W^{-n}(y))\delta_{\sigma_W^{-n}(y)}
&\text{if $y\in W$}\\
0&\text{otherwise}.
\end{cases}
\end{equation*}
\end{lemma}

\begin{proof}
We first consider the case $n\geq 0$. 
By Lemma \ref{compute}, 
we have 
\[
\pi_{(x,\gamma)}(t^n(f))\delta_y
=\gamma_0^n\sum_{\sigma^n(z)=y}f(z)\delta_{z}.
\]
We have $f(z)\neq 0$ only when $z\in W$. 
There exists $z\in W$ with $\sigma^n(z)=y$ 
only when $y\in W$, 
and in this case $z=\sigma_W^{-n}(y)$ 
is the only element in $W$ satisfying $\sigma^n(z)=y$. 
This proves the case $n\geq 0$. 
Next, we consider the case $n<0$. 
By Lemma \ref{compute}, 
we have 
\[
\pi_{(x,\gamma)}\big(t^{n}(f)\big)\delta_y
=\pi_{(x,\gamma)}\big(t^{-n}(\overline{f}\circ\sigma_W^{-n})^*\big)\delta_y
=\begin{cases}
\gamma_0^{n}(f\circ\sigma_W^{-n})(y)
\delta_{\sigma^{-n}(y)}& \text{if $y\in U_n$,}\\
0 & \text{otherwise.}
\end{cases}
\]
When $y \notin W$, we have $(f\circ\sigma_W^{-n})(y)=0$. 
This proves the case $n<0$. 
\end{proof}

Choose $f\in C_0(W)$ with $f\geq 0$ and $f(x_0')\neq 0$ 
and set 
\[
f_0=\frac{1}{n}\sum_{k=0}^{n-1}f\circ \sigma_W^{k}.
\]
Then we have $f_0 \geq 0$, $f_0(x_0')\neq 0$ and 
$f_0\circ \sigma_W=f_0$. 

Take $(x,\gamma)\in X'\times\T$ and $y\in \Orb(x)$ with $y \in W$. 
For $k=0,1,\ldots,p(y)-1$, 
we set $\xi_{y,k} \in H_{x}$ by 
\[
\xi_{y,k} := \frac{1}{\sqrt{p(y)}}\sum_{j=0}^{p(y)-1} 
\zeta_{p(y)}^{jk}\delta_{\sigma_W^j(y)}. 
\]
Then $\{\xi_{y,k}\}_{k=0}^{p(y)-1}$ is a basis of 
the span of $\{\delta_{\sigma_W^k(y)}\}_{k=0}^{p(y)-1}$. 

In these situation, we have the following. 

\begin{lemma}\label{Lem:tnf2}
For $k=0,1,\ldots,p(y)-1$ and $n\in\Z$, we have 
\begin{equation*}
\pi_{(x,\gamma)}\big(t^{n}(f_0)\big)\xi_{y,k}
=(\gamma\zeta_{p(y)}^k)^{n}f_0(y)\xi_{y,k}.
\end{equation*}
\end{lemma}

\begin{proof}
By Lemma~\ref{Lem:tnf1}, we have 
\begin{align*}
\pi_{(x,\gamma)}\big(t^{n}(f_0)\big)\xi_{y,k}
&=\frac{1}{\sqrt{p(y)}}\sum_{j=0}^{p(y)-1} 
\zeta_{p(y)}^{jk}\pi_{(x,\gamma)}\big(t^{n}(f_0)\big)\delta_{\sigma_W^j(y)}\\
&=\frac{1}{\sqrt{p(y)}}\sum_{j=0}^{p(y)-1} 
\zeta_{p(y)}^{jk}\gamma^{n}(f_0\circ\sigma_W^{-n})(\sigma_W^j(y))
\delta_{\sigma_W^{-n}(\sigma_W^j(y))}\\
&=\gamma^{n}\frac{1}{\sqrt{p(y)}}\sum_{j=0}^{p(y)-1} 
\zeta_{p(y)}^{jk}f_0(y)\delta_{\sigma_W^{j-n}(y)}\\
&=\gamma^{n}f_0(y)\frac{1}{\sqrt{p(y)}}\sum_{j=0}^{p(y)-1} 
\zeta_{p(y)}^{(j+n)k}\delta_{\sigma_W^{j}(y)}\\
&=(\gamma\zeta_{p(y)}^k)^{n}f_0(y)\xi_{y,k}. \qedhere
\end{align*}
\end{proof}

\begin{lemma}\label{Lem:pf}
For a trigonometric polynomial $q(z)=\sum_{n=-N}^N\alpha_nz^n$, 
we set $q(f_0) =\sum_{n=-N}^N\alpha_nt^n(f_0)$. 
Then for $k=0,1,\ldots,p(y)-1$, we have 
\begin{equation*}
\pi_{(x,\gamma)}\big(q(f_0)\big)\xi_{y,k}
=q(\gamma\zeta_{p(y)}^k)f_0(y)\xi_{y,k}. 
\end{equation*}
\end{lemma}

\begin{proof}
This follows from Lemma~\ref{Lem:tnf2}. 
\end{proof}

Take $g \in C_c(Z) \subset \C(\T)$ with $g(\gamma_0)\neq 0$. 
Choose a sequence $q_1,q_2,\ldots$ of trigonometric polynomials 
converging to $g$ uniformly on $\T$. 
We have the following. 

\begin{lemma}\label{Lem:gf}
The sequence $q_1(f_0),q_2(f_0),\ldots$ converges to an 
element $a \in C^*(\Sigma')$. 
We also have 
\begin{equation*}
\pi_{(x,\gamma)}\big(a\big)\xi_{y,k}
=g(\gamma_0\zeta_{p(y)}^k)f_0(y)\xi_{y,k}. 
\end{equation*}
for $(x,\gamma)\in X'\times\T$, $y\in \Orb(x) \cap W$ and 
$k=0,1,\ldots,p(y)-1$. 
\end{lemma}

\begin{proof}
Let $B=C^*(W,\sigma_W)=C_0(W)\rtimes_{\sigma_W}\Z$ 
be the universal \Ca generated 
by the products of the elements in $C_0(W)\subset B$ 
and a unitary $u\in M(B)$ which satisfy 
$u fu^*=f\circ \sigma_W$ for $f\in C_0(W)$ 
where $M(B)$ is the multiplier algebra of $B$. 
By the universality, 
there exists a \shom $\iota\colon B\to C^*(\Sigma')$ 
such that $\iota(f)=t^0(f)$ and $\iota(fu)=t^1(f)$. 
(We know that $\iota$ is injective 
by Proposition~\ref{Prop:GIUT}.) 
For $n \in \N$, we have $\iota(f_0u^n)=t^n(f_0)$. 
To see this for $n \geq 0$, choose $\xi_1,\ldots,\xi_n\in C_c(U)$ 
such that 
\[
f_0(x)=\xi_1(x)\xi_2(\sigma(x))\cdots \xi_n(\sigma^{n-1}(x))
\] 
for all $x \in U_n$ by Lemma~\ref{Lem:fctn2}, 
and compute 
\begin{align*}
\iota(f_0u^n)
&=\iota(\xi_1 (\xi_2\circ\sigma) \cdots (\xi_n\circ \sigma^{n-1}) u^n)\\
&=\iota(\xi_1 u \xi_2 u \cdots \xi_{n-1} u \xi_n u)\\
&=\iota(\xi_1 u) \iota(\xi_2 u) \cdots \iota(\xi_{n-1} u)\iota( \xi_n u)\\
&=t_1(\xi_1)t_1(\xi_2)\cdots t_1(\xi_{n-1})t_1(\xi_n)\\
&=t^n(f_0). 
\end{align*}
For $n \leq -1$, we also have $\iota(f_0u^n)=t^n(f_0)$ 
because 
\begin{align*}
\iota(f_0u^n)
=\iota(u^{-n}\overline{f_0})^*
&=\iota((\overline{f_0}\circ \sigma_W^{-n}) u^{-n})^*\\
&=t^{-n}(\overline{f_0}\circ \sigma_W^{-n})^*
=t^n(f_0)
\end{align*}
Thus for a trigonometric polynomial $q$, 
we have $\iota(f_0q(u))=q(f_0)$. 
Since $(f_0q_k(u))_k$ converges $f_0g(u)$ in $B$, 
$(q_k(f_0))_k$ converges $\iota(f_0g(u))$ in $C^*(\Sigma')$. 
Thus we get $a = \iota(f_0g(u))$. 
The last equation follows easily from Lemma~\ref{Lem:pf}. 
\end{proof}

\begin{proposition}
Let $a \in C^*(\Sigma')$ be as in Lemma~\ref{Lem:gf}. 
We have $\pi_{(x_0,\gamma_0)}(a)\neq 0$ and 
$\pi_{(x,\gamma)}(a)= 0$ for all $(x,\gamma) \in Y$. 
\end{proposition}

\begin{proof}
We have $\pi_{(x_0,\gamma_0)}(a)\neq 0$ because 
$\pi_{(x_0,\gamma_0)}(a)\xi_{x_0,0}
=g(\gamma_0)f_0(x_0)\xi_{x_0,0}\neq 0$ 
by Lemma~\ref{Lem:gf}. 

Take $(x,\gamma) \in Y$. 
Take $y \in \Orb(x)$. 
If $y \notin W$, then we can see 
$\pi_{(x,\gamma)}(a)\delta_y=0$ by Lemma~\ref{Lem:tnf1}. 
Suppose $y \in W$. 
By the condition (ii) in Definition~\ref{Def:adm}, 
we have $(y,\gamma) \in Y$. 
By the condition (iii) in Definition~\ref{Def:adm}, 
we have $(y,\zeta_{p(y)}^k\gamma) \in Y$ 
for $k=0,1,\ldots, p(y)-1$. 
Hence $\zeta_{p(y)}^k\gamma \notin Z$ for $k=0,1,\ldots, p(y)-1$. 
By Lemma~\ref{Lem:gf} 
we have 
$\pi_{(x,\gamma)}(a)\xi_{y,k}=0$ for $k=0,1,\ldots,p(y)-1$. 
Hence we get $\pi_{(x,\gamma)}(a)\delta_y=0$. 
These show that $\pi_{(x,\gamma)}(a)=0$. 
\end{proof}

We have shown the following. 

\begin{proposition}\label{Prop:YIY}
For an admissible subset $Y$ 
of $X\times\T$, 
we have $Y_{I_Y}=Y$. 
\end{proposition}

\begin{theorem}\label{MainThm}
The set of ideals of $C^*(\Sigma)$ corresponds bijectively 
to the set of all admissible subsets of $X\times\T$ 
via the maps $I\mapsto Y_I$ and $Y\mapsto I_Y$. 
\end{theorem}

\begin{proof}
This follows from Proposition~\ref{Prop:IYI}, 
Proposition~\ref{Prop:YIinv} and 
Proposition~\ref{Prop:YIY}. 
\end{proof}

\appendix

\section{Uniqueness theorems}\label{Sec:UT}

Let $\Sigma=(X,\sigma)$ be an SGDS. 
By the universality of $C^*(\Sigma)$, 
we have an action $\beta$ of the group 
$\T := \{z \in \C\mid |z|=1\}$ on $C^*(\Sigma)$ 
so that for $z \in \T$ we have 
$\beta_z(t^0(f))=t^0(f)$ for $f \in C_0(X)$ 
and $\beta_z(t^1(\xi))=zt^1(\xi)$ for $\xi \in C_c(U)$. 
This action $\beta$ is called the {\em gauge} action 
(see \cite[Section~4]{K1}). 

\begin{proposition}\label{Prop:GIUT}
Let $A$ be a \CA .
A \shom $\varPhi\colon C^*(\Sigma) \to A$ is injective 
if $\varPhi \circ t^0$ is injective and 
there exists an action of $\T$ on $A$ 
with which and the gauge action on $C^*(\Sigma)$, 
$\varPhi$ is equivariant. 
\end{proposition}

\begin{proof}
This follows from \cite[Theorem~4.5]{K1}. 
\end{proof}

\begin{definition}[{cf.\ \cite[Definition~2.5]{R2}}]\label{Def:esfree}
An SGDS $\Sigma$ is said to be {\em essentially free} 
if the interior of the set $\{x \in X\mid l(x)=0\}$ 
is empty. 
\end{definition}

By Baire's category theorem, 
$\Sigma$ is essentially free if and only if 
the interior of the set $\{x \in X\mid \sigma^n(x)=x\}$ 
is empty for every positive integer $n$. 
Thus this definition coincides with \cite[Definition~2.5]{R2}.

\begin{proposition}\label{Prop:CKUT}
Suppose an SGDS $\Sigma$ is essentially free. 
Let $A$ be a \CA .
A \shom $\varPhi\colon C^*(\Sigma) \to A$ is injective 
if $\varPhi \circ t^0$ is injective
\end{proposition}

\begin{proof}
An SGDS $\Sigma$ is essentially free 
if and only if the associate topological graph $E=(X,U,\sigma,\iota)$ 
is topologically free in the sense of \cite[Definition~5.4]{K1}. 
Thus this proposition follows from \cite[Theorem~5.12]{K1}. 
\end{proof}


\begin{thebibliography}{K2}

\bibitem[B]{B}
Blackadar, B. {\it Operator algebras. Theory of C*-algebras and von Neumann algebras.} Encyclopaedia of Mathematical Sciences, {\bf 122}. Operator Algebras and Non-commutative Geometry, III. Springer-Verlag, Berlin, 2006.

\bibitem[K1]{K1}
Katsura, T. {\it A class of $C^*$-algebras generalizing both graph algebras and homeomorphism $C^*$-algebras I, fundamental results.} Trans. Amer. Math. Soc. {\bf 356} (2004), no. 11, 4287-4322.

\bibitem[K2]{K2}
Katsura, T. {\it A class of $C^*$-algebras generalizing both graph algebras and homeomorphism $C^*$-algebras II, examples.} Internat.\ J.\ Math.\ {\bf 17} (2006), No. 7, 791--833. 

\bibitem[K3]{K3}
Katsura, T. {\it A class of $C^*$-algebras generalizing both graph algebras and homeomorphism $C^*$-algebras III, ideal structures.} Ergodic Theory Dynam. Systems {\bf 26} (2006), no. 6, 1805--1854. 

\bibitem[K4]{Ksh}
Katsura, T. {\it Topological graphs and singly generated dynamical systems.} Preprint 2021, arXiv:2107.01389, submitted. 

\bibitem[KL]{KL}
Kumjian, A.; Li, H. {\it Twisted topological graph algebras are twisted groupoid C*-algebras.} J. Operator Theory, {\bf 78} (2017), no. 1, 201--225.


\bibitem[R]{R2}
Renault, J. {\it Cuntz-like algebras.} Operator theoretical methods, 371--386, Theta Found., Bucharest, 2000.

\bibitem[SW]{SW}
Sims, A.; Williams, D. P. {\it The primitive ideals of some \'etale groupoid C*-algebras.} Algebr. Represent. Theory 19 (2016), no. 2, 255--276.

\bibitem[T1]{T1}
Tomiyama, J. {\it The interplay between topological dynamics and theory of $C\sp *$-algebras.} Lecture Notes Series, 2. Seoul National University, Research Institute of Mathematics, Global Analysis Research Center, Seoul, 1992.

\bibitem[W]{W}
Williams, D. P. {\it The topology on the primitive ideal space of transformation group $C\sp{*} $-algebras and C.C.R. transformation group $C\sp{*} $-algebras.} Trans. Amer. Math. Soc. {\bf 266} (1981), no. 2, 335--359.

\bibitem[Y]{Yee}
Yeend, T. {\it Topological higher-rank graphs and the $C^*$-algebras of topological 1-graphs.} Operator theory, operator algebras, and applications, Amer. Math. Soc., Providence, RI, 2006, 231--244.

\end{thebibliography}
\end{document}